\documentclass[11pt,reqno]{amsart}
\usepackage{amsmath,amsthm,amsfonts,epsfig,color}

\usepackage[latin1]{inputenc}

\newtheorem{theorem}{Theorem}[section]
\newtheorem{lemma}[theorem]{Lemma}
\newtheorem{proposition}[theorem]{Proposition}
\newtheorem{definition}[theorem]{Definition}
\newtheorem{remark}[theorem]{Remark}
\newtheorem{corollary}[theorem]{Corollary}

\def\deg{\textrm{deg}\,}

\def\e{\varepsilon}
\def\dist{{\rm dist}}
\def\sym{{\mathrm{sym\,}}}

\def\supp{{\mathrm{supp\,}}}
\def\S{\mathbb{S}}
\def\R{\mathbb{R}}

\def\N{\mathbb{N}}

\begin{document}

\title[$C^{1,\alpha}$ isometric embeddings]{$h$-principle and rigidity
for $C^{1,\alpha}$ isometric embeddings}

\author{Sergio Conti}
\address{Institut f\"ur Angewandte Mathematik, Universit\"at Bonn, D-53115 Bonn}
\email{sergio.conti@uni-bonn.de}

\author{Camillo De Lellis}
\address{Institut f\"ur Mathematik, Universit\"at Z\"urich, 
CH-8057 Z\"urich}
\email{camillo.delellis@math.unizh.ch}

\author{L\'aszl\'o Sz\'ekelyhidi Jr.}
\address{Hausdorff Center for Mathematics, Universit\"at Bonn, 
D-53115 Bonn}
\email{laszlo.szekelyhidi@hcm.uni-bonn.de}

\begin{abstract}
In this paper we study the embedding of Riemannian manifolds in low codimension. The well-known result of Nash and Kuiper \cite{Nash54,Kuiper55} says that any short embedding in codimension one can be uniformly approximated by $C^1$ isometric embeddings. This statement clearly cannot be true for $C^2$ embeddings in general, due to the classical rigidity in the Weyl problem. In fact Borisov extended the latter to embeddings of class $C^{1,\alpha}$ with $\alpha>2/3$ in \cite{BorisovRigidity1,BorisovRigidity2}. On the other hand he announced in \cite{Borisov65} that the Nash-Kuiper statement can be extended to local $C^{1,\alpha}$ embeddings with $\alpha<(1+n+n^2)^{-1}$, where $n$ is the dimension of the manifold, provided the metric is analytic. Subsequently a proof of the 2-dimensional case appeared in \cite{Borisov2004}. In this paper we provide analytic proofs of all these statements, for general dimension and general metric.  
\end{abstract}

\maketitle

\section{Introduction}
Let $M^n$ be a smooth compact manifold of dimension $n\geq 2$, equipped with a Riemannian metric $g$. An isometric immersion
of $(M^n,g)$ into $\R^m$ is a map $u\in C^1(M^n;\R^m)$ such that the induced metric agrees with $g$. In local coordinates this amounts to 
the system
\begin{equation}\label{e:equations}
\partial_iu\cdot\partial_ju=g_{ij}
\end{equation}
consisting of $n(n+1)/2$ equations in $m$ unknowns. If in addition $u$ is injective, it is an isometric embedding. 
Assume for the moment that $g\in C^{\infty}$. The two classical theorems concerning the solvability of this system are: 
\begin{enumerate}
\item[(A)] if $m\geq (n+2)(n+3)/2$, then any short embedding can be uniformly approximated by  isometric embeddings of class $C^{\infty}$ (Nash \cite{Nash56}, Gromov \cite{Gromov86});
\item[(B)] if $m\geq n+1$, then any short embedding can be uniformly approximated by isometric embeddings of class $C^1$ (Nash \cite{Nash54}, Kuiper \cite{Kuiper55}).
\end{enumerate}
Recall that a short embedding is an injective map $u:M^n\to \R^m$ such that the metric induced on $M$ by $u$ is shorter than $g$. In coordinates this means that $(\partial_iu\cdot\partial_ju)\leq (g_{ij})$ in the sense of quadratic forms. Thus, (A) and (B) are not merely existence theorems, they show that there exists a huge (essentially $C^0$-dense) set of solutions. This type of abundance of solutions is a central aspect of Gromov's $h$-principle, for which the isometric embedding problem is a primary example (see \cite{Gromov86,Eliashberg}). 

Naively, this type of flexibility could be expected for high codimension as in (A), since then there are many more unknowns than equations in \eqref{e:equations}. The $h$-principle for $C^1$ isometric embeddings is on the other hand rather striking, especially when compared to the classical rigidity result concerning the Weyl problem: if $(S^2,g)$ is a 
compact Riemannian surface with positive Gauss curvature and $u\in C^2$ is an isometric immersion into $\R^3$, then $u$ is uniquely determined up to a rigid motion (\cite{CohnVossen,Herglotz}, see also \cite{Spivak5} for a thorough discussion). Thus it is clear that isometric immersions have a completely different qualitative behaviour 
at low and high regularity (i.e. below and above $C^2$).

This qualitative difference is further highlighted by the following optimal mapping properties in the case when $m$ is allowed to be sufficiently high:
\begin{enumerate}
\item[(C)] if $g\in C^{l,\beta}$ with $l+\beta>2$ and $m$ is sufficiently large, then there exists a solution $u\in C^{l,\beta}$ (Nash \cite{Nash56}, Jacobowitz \cite{Jacobowitz});
\item[(D)] if $g\in C^{l,\beta}$ with $0<l+\beta<2$ and $m$ is sufficiently large, then there exists a solution $u\in C^{1,\alpha}$ with 
$\alpha<(l+\beta)/2$ (K\"allen \cite{Kallen}).
\end{enumerate}
These results are optimal in the sense that in both cases there exists $g\in C^{l,\beta}$ to which no solution $u$ has better regularity than stated.

The techniques are also different: whereas the proofs of (A) and (C) rely on the Nash-Moser implicit function theorem, the proofs of (B) and (D) involve an iteration technique called convex integration. This technique was developed by Gromov \cite{Gromov73,Gromov86} into a very powerful tool to prove the $h$-principle in a wide variety of geometric problems (see also \cite{Eliashberg,Spring}). In general the regularity of solutions obtained using convex integration agrees with the highest derivatives appearing in the equations (see \cite{SpringRegularity}). Thus, an interesting question raised in \cite{Gromov86} p219 is 
how one could extend the methods to produce more regular solutions. 
Essentially the same question, in the case of isometric embeddings,
is also mentioned in \cite{Yau} (see Problem 27).
For high codimension this is resolved in (D). 

Our primary aim in this paper is to consider the low codimension case, i.e. when $m=n+1$. 
This range was first considered by Borisov. In \cite{Borisov65} it was announced that if $g$ is analytic, then the $h$-principle holds for local isometric embeddings $u\in C^{1,\alpha}$ for $\alpha<\frac{1}{1+n+n^2}$. A  proof for the case $n=2$ appeared in \cite{Borisov2004}.
Our main result is to provide a proof of the $h$-principle in this range for $g$ which is not necessarily analytic and general $n\geq 2$ (see Section \ref{s:results1} for precise statements). Moreover, at least for $l=0$ and sufficiently small $\beta>0$, we recover the optimal mapping range corresponding to (D). Thus, there seems to be a direct trade-off between codimension and regularity.

The novelty of our approach, compared to Borisov's, is that only a finite number of derivatives need to be controlled. This is achieved by introducing a smoothing operator in the iteration step, analogous to the device of Nash used to overcome the loss of derivative problem in \cite{Nash56}. A similar method was used by K\"allen in \cite{Kallen}. See Section \ref{s:overview} for an overview of the iteration procedure. In addition, the errors coming from the smoothing operator are controlled by using certain commutator estimates on convolutions. These estimates are in Section \ref{s:convolutions}.

Concerning rigidity in the Weyl problem, it is known from the work of 
Pogorelov and Sabitov that 
\begin{enumerate}
\item closed $C^1$ surfaces with positive Gauss curvature 
and bounded extrinsic curvature are convex (see \cite{Pogorelov73});
\item closed convex surfaces are rigid in the sense that isometric 
immersions are unique up to rigid motion \cite{PogorelovRigidity};
\item a convex surface with metric $g\in C^{l,\beta}$ with $l\geq 2,0<\beta<1$ and positive curvature is of class $C^{l,\beta}$ (see \cite{Pogorelov73,Sabitov}).
\end{enumerate}
Thus, extending the rigidity in the Weyl problem to $C^{1,\alpha}$ isometric immersions can be reduced to showing that the image of the surface has bounded extrinsic curvature (for definitions see Section \ref{s:rigidity}). Using geometric arguments, in a series of papers \cite{Borisov58-1,Borisov58-2,BorisovRigidity1,Borisov58-3,BorisovRigidity2} Borisov proved that for $\alpha>2/3$ the image of surfaces with positive Gauss curvature has indeed bounded extrinsic curvature. Consequently, rigidity holds in this range and in particular $2/3$ is an upper bound on the range of H\"older exponents that can be reached using convex integration. 

Using the commutator estimates from Section \ref{s:convolutions}, at the end of this paper (in Section \ref{s:rigidity}) we provide a short and self-consistent analytic proof of this result.

\subsection{The $h$--principle for small exponents}\label{s:results1}

In this subsection we state our main existence
results for 
$C^{1, \alpha}$ isometric immersions. 
One is of local nature, whereas the second
is global. Note that for the local result 
the exponent matches
the one announced in \cite{Borisov65}.
In what follows, we denote by $\sym^+_n$ the cone
of positive definite symmetric $n\times n$ matrices. 
Moreover, given
an immersion $u: M^n\to \R^m$, we denote by $u^\sharp e$
the pullback of the standard Euclidean metric through $u$, so that in local coordinates
$$
(u^\sharp e)_{ij}=\partial_iu\cdot\partial_ju.
$$
Finally, let
$$
n_*=\frac{n(n+1)}{2}.
$$

\begin{theorem}[Local existence]\label{t:local}
Let $n\in \N$ and $g_0\in \sym^+_n$. There
exists $r>0$ such that the following holds for
any smooth bounded open set $\Omega\subset\R^n$
and any Riemannian metric $g\in C^{\beta}(\overline{\Omega})$ with 
$\beta>0$ and $\|g-g_0\|_{C^0}\leq r$. 
There exists a constant $\delta_0>0$ such that, if 
$u\in C^2(\overline{\Omega};\R^{n+1})$ and $\alpha$ satisfy
$$
\|u^\sharp e-g\|_0\leq \delta_0^2 \qquad \mbox{and}\qquad
0<\alpha<\min\left\{\frac{1}{1+2n_*},\frac{\beta}{2}\right\}\, ,
$$ 
then there exists a map $v\in C^{1,\alpha} (\overline{\Omega}; \R^{n+1})$ with 
$$
v^\sharp e \;=\; g\qquad\mbox{and}\qquad
\|v-u\|_{C^1} \;\leq\; C\, \|u^\sharp e-g\|_{C^0}^{1/2}\, .
$$
\end{theorem}

\begin{corollary}[Local h-principle]\label{c:local}
Let $n,g_0,\Omega,g,\alpha$ be as in Theorem \ref{t:local}. Given any
short map $u\in C^1(\overline{\Omega};\R^{n+1})$ and any $\e>0$ there exists an isometric immersion $v\in C^{1,\alpha}(\overline{\Omega};\R^{n+1})$ with $\|u-v\|_{C^0}\leq \e$.
\end{corollary}

\begin{theorem}[Global existence]\label{t:global}
Let $M^n$ be a smooth, compact manifold with a Riemannian metric 
$g\in C^{\beta}(M)$ and let $m\geq n+1$. There is a constant 
$\delta_0>0$ such that, if
$u\in C^2(M;\R^m)$ and $\alpha$ satisfy
$$
\|u^\sharp e-g\|_{C^0}\leq \delta_0^2 \qquad\mbox{and}\qquad
0<\alpha<\min\left\{\frac{1}{1+2(n+1)n_*},\frac{\beta}{2}\right\}\, ,
$$ 
then there exists a map $v\in C^{1,\alpha} (M; \R^{m})$ with 
$$
v^\sharp e \;=\; g\qquad\mbox{and}
\qquad \|v-u\|_{C^1} \;\leq\; C\|u^\sharp e-g\|_{C^0}^{1/2}\, .
$$
\end{theorem}

\begin{corollary}[Global $h$--principle]\label{c:global}
Let $(M^n,g)$ and $\alpha$ be as in Theorem \ref{t:global}. Given any
short map $u\in C^1(M;\R^{m})$ with $m\geq n+1$ and any $\e>0$ there exists 
an isometric immersion $v\in C^{1,\alpha}(M;\R^{m})$ with $\|u-v\|_{C^0}\leq \e$.
\end{corollary}

\begin{remark}\label{r:emb} 
In both corollaries, if $u$ is an embedding,
then there exists a corresponding $v$ which in addition
is an embedding. 
\end{remark}

\subsection{Rigidity for large exponents}\label{s:results2}
The following is a crucial estimate on the metric
pulled back by standard regularizations of a given map.

\begin{proposition}
[Quadratic estimate]\label{p:improv}
Let $\Omega\subset \R^n$ be an open set, 
$v\in C^{1, \alpha} (\Omega, \R^m)$ with
$v^\sharp e \in C^{2}$ and
$\varphi\in C^\infty_c (\R^n)$ a standard symmetric
convolution kernel. Then, for every compact set 
$K\subset \Omega$, 
\begin{equation}\label{e:improv}
\|(v*\varphi_\ell)^\sharp e - v^\sharp e\|_{C^{1} (K)} \;=\;  
O (\ell^{2\alpha -1}).
\end{equation}
\end{proposition}

In particular, fix a map $u$ and a kernel
$\varphi$ satisfying the assumptions of
the Proposition with $\alpha>1/2$.
Then the Christoffel symbols of $(v*\varphi_\ell)^\sharp e$
converge to those of $v^\sharp e$. This corresponds to the results of Borisov in
\cite{Borisov58-1,Borisov58-2}, and hints at the
absence of $h$--principle for $C^{1,\frac{1}{2}+\e}$ immersions.
Relying mainly on this estimate
we can give a fairly short proof of Borisov's theorem:

\begin{theorem}\label{t:2/3}
Let $(M^2,g)$ be a surface with $C^2$ metric and positive Gauss curvature, and let
$u\in C^{1,\alpha}(M^2;\R^3)$ be an isometric immersion 
with $\alpha>2/3$. Then $u (M)$ is a surface
of bounded extrinsic curvature.
\end{theorem}

This leads to the following corollaries, which follow from the 
work of Pogorelov and Sabitov. 

\begin{corollary}\label{c:2/3}
Let $(S^2, g)$ be a closed surface with $g\in C^2$ and positive Gauss curvature, and let
$u\in C^{1,\alpha} (S^2; \R^3)$ be an isometric immersion
with $\alpha>2/3$. Then, $u (S^2)$ is the boundary of
a bounded convex set and any two such images are congruent.
In particular if the Gauss curvature is constant,
then $u(S^2)$ is the boundary of a ball $B_r (x)$.
\end{corollary}

\begin{corollary}\label{c:localconvexity}
Let $\Omega\subset\R^2$ be open and $g\in C^{2,\beta}$ a metric on $\Omega$ with
positive Gauss curvature. Let $u\in C^{1,\alpha}(\Omega;\R^3)$ be an isometric immersion 
with $\alpha>2/3$. Then $u(\Omega)$ is $C^{2,\beta}$ and locally 
uniformly convex (that is, for every $x\in \Omega$ there 
exists a neighborhood $V$
such that $u(\Omega)\cap V$ is the graph of a $C^{2, \beta}$ function with
positive definite second derivative).
\end{corollary}

\subsection{Connections to the Euler equations}
There is an interesting analogy between isometric immersions in low codimension (in particular the Weyl problem) and the incompressible Euler equations. In \cite{us} a method, which is very closely related to convex integration, was introduced 
to construct highly irregular energy-dissipating solutions of the Euler equations. Being in conservation form, the ''expected'' regularity space for convex integration for the Euler equations should be $C^0$. This is still beyond reach, and in \cite{us} a weak version of convex integration was applied instead, to produce solutions in $L^{\infty}$ (see also \cite{us2} for a slightly better space) and, moreover, to show that a weak version of the $h$-principle holds. 

Nevertheless, just like for isometric immersions, for the Euler equations there is particular interest to go beyond $C^0$: in \cite{Onsager} L.~Onsager, motivated by the phenomenon of anomalous dissipation in turbulent flows, conjectured that there exist weak solutions of the Euler equations of class $C^{\alpha}$ with $\alpha<1/3$ which dissipate energy, whereas for $\alpha>1/3$ the energy is conserved. The latter was proved in \cite{Eyink,CET}, but on the construction of energy-dissipating weak solutions nothing is known beyond $L^{\infty}$ (for previous work see \cite{Scheffer, Shnirelman1,Shnirelman2}). It should be mentioned that the critical exponent $1/3$ is very natural - it agrees with the scaling of the energy cascade predicted by Kolmogorov's theory of turbulence (see for instance \cite{Frisch}). 

For the analogous problem for isometric immersions there does
not seem to be a universally accepted critical exponent (c.f. Problem 27 of \cite{Yau}), even though 1/2 seems likely (c.f. section \ref{s:results2} and the discussion in \cite{Borisov2004}).
In fact, the regularization and the commutator estimates used
in our proof of Proposition \ref{p:improv} and Theorem \ref{t:2/3} have been inspired by
(and are closely related to)
the arguments of \cite{CET}.

\section{Estimates on convolutions: Proof of
Proposition \ref{p:improv}}\label{s:convolutions}

As usual, we denote the norm on the H\"older space $C^{k,\alpha}(\overline{\Omega})$ by 
$$
\|f\|_{k,\alpha}:=\sup_{x\in\Omega}\sum_{|a|\leq k}|\partial^af(x)|+\sup_{x,y\in\Omega,\,x\neq y}\sum_{|a|=k}\frac{|\partial^af(x)-\partial^af(y)|}{|x-y|^{\alpha}}\,.
$$
Here $k=0,1,2,\dots$, $a=(a_1,\dots,a_n)$ is a multi-index with $|a|=a_1+\dots+a_n$ and $\alpha\in [0,1[$. For simplicity we will also use the abbreviation
$\|f\|_{k}=\|f\|_{k,0}$ and $\|f\|_{\alpha}=\|f\|_{0,\alpha}$. 

Recall the following interpolation inequalities for these norms:
$$
\|f\|_{k,\alpha}\leq C\|f\|_{k_1,\alpha_1}^{\lambda}\|f\|_{k_2,\alpha_2}^{1-\lambda},
$$
where $C$ depends on the various parameters, $0<\lambda<1$ and 
$$
k+\alpha=\lambda(k_1+\alpha_1)+(1-\lambda)(k_2+\alpha_2).
$$
The following estimates are well known and play 
a fundamental role in both the constructions and the proof of rigidity.
\begin{lemma}\label{l:mollify}
Let  $\varphi\in C^\infty_c(\R^n)$ be symmetric and such that $\int\varphi=1$. Then for any $r,s\geq 0$ and $\alpha\in ]0,1]$ we have
\begin{equation}\label{e:mollify1}
\|f*\varphi_\ell\|_{r+s}\leq C\ell^{-s}\|f\|_r,
\end{equation} 
\begin{equation}\label{e:mollify2}
\|f-f*\varphi_\ell\|_r \leq C\ell^{2}\|f\|_{r+2},
\end{equation}
\begin{equation}\label{e:mollify3}
\|(fg)*\varphi_\ell-(f*\varphi_\ell)(g*\varphi_\ell)\|_r\leq C\ell^{2\alpha -r}\|f\|_\alpha\|g\|_\alpha.
\end{equation}
\end{lemma}

\begin{proof}
For any multi-indices $a,b$ with $|a|=r, |b|=s$ we have $\partial^{a+b}(f*\varphi_\ell)=\partial^{a} f*\partial^{b} \varphi_\ell$, hence
$$
|\partial^{a+b}(f*\varphi_\ell)|\leq C_s\ell^{-s}\|f\|_r.
$$
This proves \eqref{e:mollify1}. 

Next, by considering the Taylor expansion of $f$ at $x$ we see that
$$
f(x-y)-f(x)=f'(x)y+r_x(y),
$$
where $\sup_x|r_x(y)|\leq C|y|^2\|f\|_2$. Moreover, since $\varphi$ is symmetric, 
$$
\int \varphi_\ell(y) y\, dy \;=\; 0\, .
$$ 
Thus,
\begin{equation*}
\begin{split}
|f-f*\varphi_\ell|&\;=\; \left|\int \varphi_\ell(y)(f(x-y)-f(x))dy\right|\\
&\;\leq\; C\|f\|_2\int \ell^{-n}\left|\varphi\left(\frac{y}{\ell}\right)\right||y|^{2}dy
\quad=\quad C\ell^{2}\|f\|_2\, .
\end{split}
\end{equation*}
This proves \eqref{e:mollify2} for the case $r=0$. To obtain the estimate for general $r$, repeat the same argument for the partial derivatives
$\partial^a f$ with $|a|=r$.

For the proof of estimate \eqref{e:mollify3} let $a$ be any multi-index 
with $|a|=r$. By the product rule
\begin{equation*}
\begin{split}
\partial^{a}\bigl[\varphi_{\ell}*(fg)-&(\varphi_{\ell}*f)(\varphi_{\ell}*g)\bigr]=\\
=&\partial^{a}\varphi_{\ell}*(fg)-\sum_{b\leq a}\begin{pmatrix}a\\ b\end{pmatrix}(\partial^{b}\varphi_{\ell}*f)(\partial^{a-b}\varphi_{\ell}*g)\\
=&\partial^{a}\varphi_{\ell}*(fg)-(\partial^{a}\varphi_{\ell}*f)(\varphi_{\ell}*g)+(\varphi_{\ell}*f)(\partial^{a}\varphi_{\ell}*g)\\
&-\sum_{0<b<a}\begin{pmatrix}a\\ b\end{pmatrix}[\partial^{b}\varphi_{\ell}*(f-f(x))][\partial^{a-b}\varphi_{\ell}*(g-g(x))]\\
=&\partial^{a}\varphi_{\ell}*[(f-f(x))(g-g(x))]\\
&-\sum_{b\leq a}\begin{pmatrix}a\\ b\end{pmatrix}\partial^{b}\varphi_{\ell}*(f-f(x))\cdot\partial^{a-b}\varphi_{\ell}*(g-g(x)),
\end{split}
\end{equation*}
where we have used the fact that 
$$
\partial^{a}\varphi_\ell*f(x)=\begin{cases}
f(x)&\textrm{ if }a=0,\\
0&\textrm{ if }a\neq 0.
\end{cases}
$$
Now observe that
\begin{equation*}
\begin{split}
|\partial^{a}\varphi_{\ell}*&[(f-f(x))(g-g(x))]|\\
&=\;\left|\int \partial^{a}\varphi_\ell(y)(f(x-y)-f(x))(g(x-y)-g(x))dy\right|\\
&\leq\; \int |\partial^{a}\varphi_\ell(y)||y|^{2\alpha}dy\,\|f\|_{\alpha}\|g\|_{\alpha}
\;=\; C_r\,\ell^{2\alpha-r}\|f\|_{\alpha}\|g\|_{\alpha}.
\end{split}
\end{equation*}
Similarly, all the terms in the sum over $b$ obey the same estimate. This 
concludes the proof of \eqref{e:mollify3}.
\end{proof}

\begin{proof}[Proof of Proposition \ref{p:improv}]
Set $g:= v^\sharp e$ and $g^\ell:= (v*\varphi_\ell)^\sharp e$. We have
$$
\|g^\ell_{ij} - g_{ij}\|_{1}\;\leq\; 
\|g^\ell_{ij} - g_{ij}* \varphi_\ell\|_{1}
+ \|g_{ij}*\varphi_\ell - g_{ij}\|_{1}\, .
$$
The first term can be written as
\begin{equation*}
\|g^\ell_{ij} - g_{ij}*\varphi_\ell\|_{1} =
\left\|\partial_j v*\varphi_\ell \cdot 
\partial_i v * \varphi_\ell 
- (\partial_j v \cdot \partial_i v)*\varphi_\ell\right\|_{1},
\end{equation*}
so that \eqref{e:mollify3} applies, to yield the bound $\ell^{2\alpha-1}\|v\|_{1,\alpha}^2$. For the second term \eqref{e:mollify2} gives the bound
$\ell\|g\|_{2}$. Combining these two we obtain  
$$
\|g^\ell_{ij} - g_{ij}\|_{k}\;\leq\;
C(\ell^{2\alpha-1}\|v\|_{1,\alpha}^2+\ell\|g\|_{2})\, ,
$$
from which \eqref{e:improv} readily follows.
\end{proof}

\section{$h$--principle: The general scheme}\label{s:overview}

The general scheme of our construction follows the method of Nash and Kuiper \cite{Nash54,Kuiper55}. For convenience of the reader we sketch this scheme in this section. Assume for simplicity that $g$ is smooth.

The existence theorems are based on an iteration of {\it stages}, and each {\it stage} consists of several {\it steps}. The purpose of a {\it stage} is to correct the error $g-u^\sharp e$. In order to achieve this correction, the error is decomposed into a sum of primitive metrics as
\begin{eqnarray*}
g-u^\sharp e &=& \sum_{k=1}^{n_*}a_k^2\nu_k\otimes\nu_k\quad\textrm{ (locally)}\\
g-u^\sharp e &=& \sum_j\sum_{k=1}^{n_*}(\psi_ja_{j,k})^2\nu_{j,k}\otimes\nu_{j,k}\quad\textrm{ (globally)}
\end{eqnarray*}
The natural estimates associated with this decomposition are
\begin{equation*}
\begin{split}
\|&a_k\|_0\;\sim\;\|g-u^\sharp e\|_0^{1/2}\\
\|&a_k\|_{N+1}\;\sim\;\,\|u\|_{N+2}\quad\textrm{ for }N=0,1,2,\dots.
\end{split}
\end{equation*}
A {\it step} then involves adding one primitive metric. In other words the goal of a {\it step} is the metric change
$$
u^\sharp e\quad\mapsto\quad u^\sharp e+a^2\nu\otimes\nu.
$$
Nash used spiralling perturbations (also known as the "Nash twist") to achieve this; for the codimension one case Kuiper replaced the spirals by corrugations. Using the same ansatz (see formula \eqref{e:ansatz}) one easily checks that addition of a primitive metric is possible with the following estimates (see Proposition \ref{p:step}):
\begin{eqnarray*}
\textrm{$C^0$-error in the metric}\;&\sim&\;\|g-u^\sharp e\|_0\,\frac{1}{K}\\
\textrm{ increase of $C^1$-norm of $u$}\;&\sim&\;\|g-u^\sharp e\|_0^{1/2}\\
\textrm{ increase of $C^2$-norm of $u$}\;&\sim&\;\|u\|_2\,K
\end{eqnarray*}
for any $K\geq 1$. Observe that the first two of these estimates is essentially the same as in \cite{Nash54,Kuiper55}. Furthermore, the third estimate is only valid modulo a "loss of derivative" (see Remark \ref{r:lossofderivative}). 

The low codimension forces the steps to be performed serially. This is in contrast with the method of K\"allen in \cite{Kallen}, where the whole {\it stage} can be performed in one step due to the high codimension. Thus the number of {\it steps} in a {\it stage} equals the number of primitive metrics in the above decomposition which interact. This equals $n_*$ for the local construction and $(n+1)n_*$ for the global construction. 
To deal with the "loss of derivative" problem we mollify the map $u$ at the start of every stage, in a similar manner as is done in a Nash-Moser iteration. Because of the quadratic estimate \eqref{e:mollify3} in Lemma \ref{l:mollify} there will be no additional error coming from the mollification. 
Therefore, iterating the estimates for one step over a single stage (that is, over $N_*$ steps) leads to 
\begin{eqnarray*}
\textrm{$C^0$-error in the metric}\;&\sim&\;\|g-u^\sharp e\|_0\,\frac{1}{K}\\
\textrm{ increase of $C^1$-norm of $u$}\;&\sim&\;\|g-u^\sharp e\|_0^{1/2}\\
\textrm{ increase of $C^2$-norm of $u$}\;&\sim&\;\|u\|_2\,K^{N_*}
\end{eqnarray*}
With these estimates, iterating over the {\it stages} leads to exponential convergence of the metric error, leading to a controlled growth of the $C^1$ norm and an exponential growth of the $C^2$ norm of the map. In particular, interpolating between these two norms leads to convergence in $C^{1,\alpha}$ for $\alpha<\frac{1}{1+2N_*}$.

\section{$h$--principle: Construction step}\label{s:step}

The main step of our construction is given by the following
proposition. 

\begin{proposition}[Construction step]\label{p:step}
Let $\Omega\subset \R^n$, $\nu\in S^{n-1}$ and $N\in \N$. 
Let $u\in C^{N+2}(\overline{\Omega};\R^{n+1})$ and $a\in C^{N+1}(\overline{\Omega})$. Assume that
$\gamma\geq 1$ and $\ell,\delta\leq 1$ are constants such that 
\begin{eqnarray}
\frac{1}{\gamma}I\;\leq \;&u^\sharp e &\;\leq \gamma I \quad\textrm{ in }\Omega,\label{e:nondeg}\\
\|a\|_0 &\leq& \delta,\label{e:dmmu}\\
\|u\|_{k+2}+\|a\|_{k+1} &\leq& \delta\ell^{-(k+1)}\,\textrm{ for }k=0,1,\dots,N\label{e:conditiona}.
\end{eqnarray}
Then, for any
\begin{equation}\label{e:lambda}
 \lambda\;\ge\;\ell^{-1}
\end{equation}
there exists $v\in C^{N+1}(\overline{\Omega};\R^{n+1})$ such that
\begin{equation}\label{e:esth2}
\left\|v^\sharp e-(u^\sharp e + a^2  \nu\otimes \nu)\right\|_0 \le
C\;\frac{\delta^2}{\lambda\ell}
\end{equation}
and
\begin{equation}\label{e:esth3}
\left\|u-v\right\|_{j}\le C\;\delta\; \lambda^{j-1}\quad \textrm{ for }j=0,1,\dots,N+1,
\end{equation}
where $C$ is a constant depending only on $n,N$ and $\gamma$.
\end{proposition}

\begin{remark}\label{r:lossofderivative}
Observe that if \eqref{e:esth3} would hold for $j=N+2$, then the conclusion of the proposition
would say essentially (with $N=0$) that the equation
$$
v^\sharp e=u^\sharp e + a^2  \nu\otimes \nu
$$
admits approximate solutions in $C^{2}$ with estimates
\begin{eqnarray*}
\|v^\sharp e-(u^\sharp e + a^2  \nu\otimes \nu)\|_0 & \leq
&C\;\delta^2\frac{1}{K},\\
\left\|u-v\right\|_{2}&\leq& C\;\|u\|_{2}\,K.
\end{eqnarray*}
Here $K=\lambda\ell\geq 1$. The fact that \eqref{e:esth3}
holds only for $j\le N+1$ amounts to a 
"loss of derivative" in the estimate. 
\end{remark}

In the higher codimension case we need an additional
technical assumption in order to carry on the same 
result. As usual the oscillation ${\rm osc}\, u$ of a 
vector-valued map $u$
is defined as $\sup_{x,y} |u(x)-u(y)|$.

\begin{proposition}[Step in higher codim.]\label{p:step2}
Let $m, n,N\in\N$ with $n,N\ge 1$ and $m\geq n+1$. 
Then there exist a
constant $\eta_0>0$ with the following property.
Let $\Omega$, $g$, $a$, $\nu$ and $u\in C^{2+N}
(\overline{\Omega}, \R^m)$ satisfy the assumptions of
Proposition \ref{p:step} and assume in addition 
${\rm osc}\, \nabla u\leq \eta_0$. 
Then there exists a map $v\in C^{1+N} (\overline{\Omega},
\R^m)$ satisfying the same
conclusion as in Proposition \ref{p:step}.
\end{proposition}

\subsection{Basic building block}
In order to prove the Proposition we need the following lemma. The function $\Gamma$ will be
our "corrugation".

\begin{lemma}\label{l:Gamma}
There exists $\delta_*>0$ and a function
$\Gamma\in C^\infty([0,\delta_*]\times\R;\R^2)$ with
$\Gamma(\delta,t+2\pi)=\Gamma(\delta, t)$ and having the following
properties:
\begin{eqnarray}
 \left|\partial_t \Gamma(s,t) +e_1\right|^2 &=& 1 
+ s^2\, ,\label{e:Gamma1}\\
|\partial_s \partial^k_t \Gamma_1(s,t)| + |\partial_t^k\Gamma(s,t)|
&\le& C_k s\label{e:Gamma2} \qquad\qquad\mbox{for $k\geq 0$.}
\end{eqnarray}
\end{lemma}
\begin{proof}
Define $H:\R^2\to \R^2$ as
$H(\tau,t)= (\cos(\tau\sin t), \sin(\tau\sin t))$. Then
\begin{equation}\label{e:sine}
\int_0^{2\pi} H_2(\tau,t)\, dt  \;=\; \int_0^{2\pi} 
\sin(\tau\sin t)\, dt \;=\; \int_{-\pi}^\pi \sin(\tau\sin t)\, dt
\;=\; 0
\end{equation}
by the symmetry of the sine function. Set
\begin{equation}\label{e:cosine}
J_0 (\tau) \;:=\; \frac{1}{2\pi} \int_0^{2\pi} H_1(\tau,t)\, dt  
\;=\; \frac{1}{2\pi} \int_0^{2\pi} 
\cos(\tau\sin t)\, dt\,.
\end{equation}
Note that $J_0\in C^\infty(\R)$ with $J_0(0)=1$, $J_0'(0)=0$
and $J'' (0) <0$. 
We claim that there exists $\delta>0$ and 
a function $f\in C^\infty(-\delta,\delta)$ such that  
$f(0)=0$ and 
\begin{equation}\label{e:J0}
  J_0(f(s)) \;=\; \frac{1}{\sqrt{1+s^2}}\,.
\end{equation}
This is a consequence of the implicit function theorem. To see this, set
$$
F(s,r)=J_0(r^{1/2})-(1+s^2)^{-1/2}.
$$
Then $F\in C^{\infty}(\R^2)$. Indeed, 
since the Taylor expansion of $\cos x$ contains only even
powers of $x$, $J_0 (r^{1/2})$ is obviously analytic. Moreover,
$$
J_0(r^{1/2})=\frac{1}{2\pi}\int_0^{2\pi} \left( 1-\frac{r}{2}\sin^2 t\right)\, dt 
+ O(r^2).
$$
In particular $\partial_r F(0,0)=-1/4$. Since also 
$F(0,0)=0$, the implicit funcion theorem yields $\delta>0$ 
and $g\in C^{\infty}(-\delta,\delta)$ such that $g(0)=0$ and
$$
F(s,g(s))=0.
$$
Next, observe that $\partial_s F(0,0)=0$ and $\partial^2_s F(0,0)=1$. Therefore
$$
g'(0)=0\textrm{ and }g''(0)=4.
$$
This implies that $f(s):=g(s)^{1/2}$ is also a smooth function, with
$$
f(0)=0\textrm{ and }f'(0)=\sqrt{2},
$$
thus proving our claim.

Having found $f\in C^{\infty}(-\delta,\delta)$ with $f(0)=0$ and \eqref{e:J0}, we finally set
\begin{equation*}
  \Gamma(s,t)\;:=\; 
\int_0^t \left[ \sqrt{1+s^2} H(f(s), t') - e_1
  \right] dt'\,.
\end{equation*}
By construction $|\partial_t \Gamma(s,t) +e_1|^2 = 1+s^2$. Moreover
\begin{eqnarray*}
\Gamma(s,t + 2\pi) - \Gamma (s, t) &=&
\int_t^{t+2\pi} \left[
\sqrt{1+s^2} H (f(s), t') - e_1\right]\, dt'\\
&=&\sqrt{1+s^2} \int_0^{2\pi} H (f(s), t')\, dt' - 2\pi e_1\\  
&\stackrel{\eqref{e:sine}\eqref{e:cosine}}{=}& 
2\pi e_1 \left[ \sqrt{1+s^2} J_0(f(s)) - 1\right]
\;\stackrel{\eqref{e:J0}}{=}\; 0. 
\end{eqnarray*}
Thus the function $\Gamma$ is $2\pi$-periodic in the second argument.

We now come to the estimates. Fix $\delta_*<\delta$. Then
$\Gamma\in C([0,\delta_*]\times\R;\R^2)$, and since it is periodic 
 in the second variable, $\Gamma$
and all its partial derivatives are uniformly bounded.
Straightforward computations show that for any $k=0,1,\dots$
$$
\partial^k_t \Gamma (0, t) = 0\,\textrm{ and }\, 
\partial_s \partial^k_t \Gamma_1 (0,t) = 0 \qquad\mbox{for all $t$.}
$$
Hence, integrating in $s$, we conclude that
\begin{eqnarray*}
|\partial^k_t \Gamma (s, t)|&\leq&s\, \|\partial_s
\partial^k_t \Gamma\|_0\, ,\\
|\partial_s \partial^k_t \Gamma_1 (s, t)|
&\leq& s\, \|\partial^2_s\partial^k_t\Gamma_1\|_0,
\end{eqnarray*}
which give the desired estimates.

\end{proof}

\subsection{Proof of Proposition \ref{p:step}}
Throughout the proof the letter $C$ will denote a constant, whose value might change from line to line, but otherwise depends only on $n,N$ and $\gamma$.
Fix a choice of orthonormal coordinates in $\R^n$. In these coordinates the pullback metric can be written as
$(u^\sharp e)_{ij}=\partial_iu\cdot\partial_ju$
or, denoting the matrix differential of $u$ by $\nabla u=(\partial_ju^i)_{ij}$, as
$$
u^\sharp e=\nabla u^T\nabla u.
$$
From now on we will work with this notation.

Let 
\begin{equation*}
\xi=\nabla u\cdot (\nabla u^T\nabla u)^{-1}\cdot \nu,\quad\zeta = \partial_1 u \wedge \partial_2 u 
\wedge \dots \wedge \partial_n u\,.
\end{equation*}
Because of \eqref{e:nondeg} the vectorfields $\xi,\zeta$ are well-defined and satisfy
\begin{equation}\label{e:normalfield}
\frac{1}{C}\le |\xi(x)|,\,|\zeta(x)| \le C\quad \textrm{ for }x\in\Omega
\end{equation}
with some $C\ge 1$. 
Now let
$$
\xi_1=\frac{\xi}{|\xi|^2},\quad \xi_2=\frac{\zeta}{|\xi||\zeta|},\quad \Psi (x) = \xi_1(x)\otimes e_1 + \xi_2(x) \otimes e_2,
$$
and
$$
\tilde a=|\xi|a.
$$
Then
\begin{equation}\label{e:identities}
\nabla u^T\,\Psi = \frac{1}{|\xi|^2} \nu\otimes e_1,\quad \Psi^T\Psi=\frac{1}{|\xi|^2}I,
\end{equation}
and 
\begin{equation}\label{e:atilde}
\begin{split}
\|\Psi\|_{j}&\le C\|u\|_{j+1},\\
\|\tilde a\|_{j}&\le C(\|a\|_{j}+\|a\|_0\|u\|_{j+1}),
\end{split}
\end{equation}
for $j=0,1,\dots,N+1$.
Finally, let
\begin{equation}\label{e:ansatz}
  v (x)\;:=\; u(x) + \frac{1}{\lambda} \Psi(x) \Gamma\bigl( \tilde a(x),\lambda x\cdot \nu\bigr),
\end{equation}
where $\Gamma=\Gamma(s,t)$ is the function constructed in Lemma \ref{l:Gamma}.

\medskip

{\bf Proof of \eqref{e:esth2}.} First we compute $\nabla v^T\nabla v$.
We have 
\begin{equation}\label{e:abcd}
\nabla v \;=\; \underbrace{\nabla u+\Psi \cdot \partial_t \Gamma\otimes \nu}_{A} +
\underbrace{\lambda^{-1}\, \Psi\cdot \partial_s 
\Gamma\otimes \nabla \tilde a}_{E_1} + 
\underbrace{\lambda^{-1} \,  \nabla \Psi \cdot \Gamma}_{E_2}\,.
\end{equation}
Using the notation
${\rm sym} (A)=(A + A^T)/2$ one has
\begin{equation*}
\nabla v^T\nabla v = A^TA +2 {\rm  sym}(A^TE_1+A^TE_2) +  (E_1+E_2)^T(E_1+E_2).
\end{equation*}
Using \eqref{e:identities} and \eqref{e:Gamma1}:
\begin{equation*}
\begin{split}
A^TA&=\nabla u^T\nabla u + \frac{1}{|\xi|^2}(2\partial_t\Gamma_1+|\partial_t\Gamma|^2)\nu\otimes\nu\\
&=\nabla u^T\nabla u+\frac{1}{|\xi|^2}\tilde a^2 \nu\otimes\nu=\nabla u^T\nabla u+a^2\nu\otimes \nu.
\end{split}
\end{equation*}
Next we estimate the error terms. First of all
\begin{equation*}
\begin{split}
A^TE_1&=\frac{1}{\lambda}(\nabla u^T\Psi)\, (\partial_{s}\Gamma\otimes \nabla \tilde a)+
\frac{1}{\lambda}(\nu\otimes\partial_t\Gamma) (\Psi^T\Psi) (\partial_{s}\Gamma\otimes\nabla \tilde a)\\
&=\frac{1}{\lambda|\xi|^2}\left(\partial_{s}\Gamma_1+\partial_t\Gamma\cdot\partial_{s}\Gamma\right) (\nu\otimes\nabla \tilde a).
\end{split}
\end{equation*}
Note that \eqref{e:Gamma2} together with \eqref{e:atilde} implies:
$$
\|\Gamma\|_0,\,\|\partial_t\Gamma\|_0,\,\|\partial_s\Gamma_1\|_0\,\leq\, C\,\|a\|_0.
$$
Therefore
$$
\|{\rm sym}(A^T E_1)\|_0\le \frac{C}{\lambda}\|a\|_0\|\tilde a\|_1\le C\frac{\delta^2}{\lambda\ell},
$$
and similarly 
\begin{equation*}
\|{\rm sym}(A^T E_2)\|_0\le \frac{C}{\lambda}\|a\|_0\|u\|_2\le C\frac{\delta^2}{\lambda\ell}.
\end{equation*}
Finally, 
\begin{equation*}
\begin{split}
\|E_1+E_2\|_0&\le \frac{C}{\lambda}(\|\tilde a\|_1+\|a\|_0\|u\|_2)\le \frac{C}{\lambda}(\|a\|_1+\delta\;\|u\|_2)\le C\frac{\delta}{\lambda\ell}.
\end{split}
\end{equation*}
In particular $\|E_1+E_2\|_0\leq C\delta$ and hence
\begin{equation*}
\|(E_1+E_2)^T(E_1+E_2)\|_0\leq C\frac{\delta^2}{\lambda\ell}.
\end{equation*}
Putting these estimates together we obtain \eqref{e:esth2} as required.

\medskip

{\bf Proof of \eqref{e:esth3}.} In fact 
$$
\|u-v\|_0\leq C\delta \frac{1}{\lambda} 
$$
is obvious, whereas the estimates for $j=1,\dots,N$ will follow by interpolation, 
provided the case $j=N+1$ holds. Therefore, we now prove this case.
A simple application of the product rule and interpolation yields
\begin{equation*}
\begin{split}
\|v-u\|_{N+1}&\le \frac{C}{\lambda}\left(\|\Psi\|_{N+1}\|\Gamma\|_0 + \|\Psi\|_0\|\Gamma \|_{N+1}\right)\\
&\le \frac{C}{\lambda}\left(\|u\|_{N+2}\|\tilde a\|_0+\|\Gamma\|_{N+1}\right).
\end{split}
\end{equation*}
Denoting by $D_x^j$ any partial derivative in the variables $x_1,\dots,x_n$ of order $j$, the chain rule can be written symbolically as
$$
D_x^{N+1}\Gamma=\sum_{i+j\le N+1}\left(\partial_s^{i}\partial_t^{j}\Gamma\right) \lambda^{j}\sum_{\sigma}C_{i,j,\sigma}(D_x\tilde a)^{\sigma_1}(D^2_x\tilde a)^{\sigma_2}\cdot\dots\cdot(D^{N+1}_x\tilde a)^{\sigma_{N+1}},
$$
where the inner sum is over all $\sigma$ with
\begin{eqnarray*}
\sigma_1+\dots+\sigma_{N+1}&=&i,\\
\sigma_1+2\sigma_2+\dots+(N+1)\sigma_{N+1}+j&=&N+1.
\end{eqnarray*}
These relations can be checked by counting the order of differentiation. 
Therefore, by using \eqref{e:dmmu}, \eqref{e:conditiona} and \eqref{e:lambda}
\begin{equation*}
\begin{split}
\|D_x^{N+1}\Gamma\|_0&\le C\sum_{i+j\le N+1}\left\|\partial_s^{i}\partial_t^{j}\Gamma\right\|_0\,\lambda^j \delta^{i}\,\ell^{-(N+1-j)}\\
&\le C\sum_{i+j\le N+1}\left\|\partial_s^{i}\partial_t^{j}\Gamma\right\|_0\,\delta^{i}\lambda^{N+1}\;\leq\; C\delta\lambda^{N+1}.
\end{split}
\end{equation*}
In particular, since $\|\Gamma\|_0\leq \delta$, we deduce that
$\|\Gamma\|_{N+1}\leq C\delta \lambda^{N+1}$.
Therefore
\begin{equation*}
\|v-u\|_{N+1}\;\leq\; \frac{C}{\lambda}\left(\delta \|u\|_{N+2}+
\delta \lambda^{N+1}\right)\;\leq\; C\,\delta \lambda^N\, .
\end{equation*}
This concludes the proof of the proposition.

\subsection{Proof of Proposition \ref{p:step2}}
The proof of Proposition \ref{p:step} would carry
over to this case if we can choose an appropriate
normal vector field $\zeta$ as at the beginning of
the proof of Proposition \ref{p:step}, enjoying the estimate \eqref{e:normalfield} with a fixed constant.

To obtain $\zeta(x)$ let $T (x)$ be the tangent plane
to $u (\R^n)$ at the point $u (x)$, i.e.
the plane generated by $\{\partial_1 u, \ldots, \partial_n u\}$.
Denote by $\pi_x$ the orthogonal projection of $\R^m$ onto
$T (x)$. 
Assuming that $\nabla u$ has oscillation smaller than $\eta_0$,
there exists a vector $w\in S^{n-1}$ such that $|\pi_x w|\leq 1/2$
for every $x\in \overline{\Omega}$. Hence, we can define
$$
\zeta (x) \;:=\; w - \pi_x w.
$$
It is straightforward to see that this choice of $\zeta$ 
gives
a map enjoying the same estimates as the $\zeta$ used in
the proof of Proposition \ref{p:step}.


\section{$h$--principle: stage}\label{s:stage}

\begin{proposition}[Stage, local]\label{p:stage}
For all $g_0\in \sym_n^+$ there exists $0<r<1$ such that
the following holds for any $\Omega\subset\R^n$ and 
$g\in C^{\beta} (\overline{\Omega})$ with $\|g-g_0\|_0\leq r$. 
There exists a $\delta_0>0$ such that, if 
$K\geq 1$ and $u\in C^2(\overline{\Omega}, \R^{n+1})$ satisfies
$$
\|u^\sharp e - g\|_0 \;\leq\; \delta^2 \;\leq\; \delta_0^2\qquad
\mbox{and}\qquad \|u\|_2 \;\leq\; \mu\, ,
$$ 
then there exists $v\in C^2(\overline{\Omega}, \R^{n+1})$
with
\begin{eqnarray}
\|v^\sharp e - g\|_0&\leq& C \delta^2\, , 
\left(\frac{1}{K}+\delta^{\beta-2}\mu^{-\beta}\right)\label{e:localerror}\\
\|v\|_2&\leq& C \mu K^{n_*}\, ,  \label{e:localC2}\\
\|u - v\|_1&\leq& C \delta\, .\label{e:localC1}
\end{eqnarray}
Here $C$ is a constant depending only on $n,g_0,g$ and $\Omega$.
\end{proposition}

The Proposition above is the basic stage of
the iteration scheme which will prove Theorem \ref{t:local}.
A similar proposition, to be used in the proof
of Theorem \ref{t:global} will be stated later.

\subsection{Decomposing a metric into primitive metrics} 
\begin{lemma}\label{l:Carat}
Let $g_0\in \sym_n^+$. Then there exists $r>0$, 
vectors $\nu_1,\dots,\nu_{n_*}\in \S^{n-1}$
and linear maps $L_k: \sym_n\to \R$
such that 
$$
g = \sum_{k=1}^{n_*} L_k (g) \nu_k\otimes \nu_k\quad \mbox{for every $g\in \sym_n$}
$$
and, moreover,
$L_k (g)\geq r$ for every $k$ and every $g\in\sym_n^+$ with $|g-g_0|\leq r$.
\end{lemma}
\begin{proof} 
Consider the set $S:=\{(e_i+e_j)\otimes (e_i+e_j), i\leq j\}$, where $\{e_i\}$
is the standard basis of $\R^n$. Since the span of $S$
contains all matrices of the form $e_i\otimes e_j + e_j\otimes e_i$,
clearly $S$ generates $\sym_n$. On the other hand $S$ consists of
$n_*$ matrices with $n_* = {\rm dim}\, (\sym_n)$.
So $S$ is a basis for $\sym_n$. Let us relabel the vectors $e_i+e_j$ ($i\leq j$) as
$f_1,\dots,f_{n_*}$, and let 
$$
h=\sum_{k=1}^{n_*}f_k\otimes f_k.
$$
Then $h\in \sym_n^+$ and hence there exists an invertible 
linear transformation $L$ such that $L h L^T = g_0$. 
In particular, writing $\nu_k=Lf_k/|Lf_k|\in \S^{n-1}$, we have
$$
g_0 \;=\; \sum_{k=1}^{n_*} Lf_k\otimes Lf_k\, 
\;=\; \sum_{k=1}^{n_*} |Lf_k|^{2}  \nu_k\otimes \nu_k\, .
$$
Note that the set $\{\nu_k\otimes \nu_k\}$ is also a basis for 
$\sym_n$ and therefore there exist linear maps $L_k: \sym_n \to \R$
such that $\sum L_k (A) \nu_k\otimes \nu_k$ is the unique
representation of $A\in \sym_n$ as linear combination of $\nu_k\otimes \nu_k$.
In particular, $L_i (g_0) = |Lf_k|^{2} >0$. The existence of $r>0$ satisfying the
claim of the lemma follows easily.
\end{proof}

\subsection{Proof of Proposition \ref{p:stage}}
Choose $r>0$ and $\gamma>1$ so that the statement of Lemma \ref{l:Carat} holds with
$g_0$ and $2r$, and so that 
$$
\frac{1}{\gamma} I\leq h\leq \gamma 
\qquad \mbox{for any $h\in \sym_n^+$ with $|h-g_0|<2r$}.
$$
Moreover, extend $u$ and $g$ to $\R^n$ so that 
$$
\|u\|_{C^2(\R^n)}\leq C\|u\|_{C^2(\overline{\Omega})},\quad  
\|g\|_{C^\beta(\R^n)}\leq C\|g\|_{C^\beta(\overline{\Omega})}.
$$
The procedure of such an extension is well known, with the constant $C$ depending on $n,\beta$ and $\Omega$.
In what follows, the various constants will be allowed to depend in addition on $r$ and $\gamma$.

\medskip

{\bf Step 1. Mollification.} We set 
$$
\ell =\frac{\delta}{\mu},
$$ 
and let
\begin{equation*}
\tilde u = u*\varphi_\ell,\quad \tilde g = g*\varphi_\ell,
\end{equation*}
where $\varphi\in C^\infty_c (B_1 (0))$ 
is a symmetric nonnegative convolution kernel with $\int\varphi=1$.
Lemma \ref{l:mollify} implies
\begin{eqnarray}
\|\tilde u-u\|_1&\leq& C\|u\|_2\,\ell \leq C\delta,\label{e:localmollifiedu}\\
\|\tilde g-g\|_0&\leq& C \|g\|_\beta\,\ell^\beta,\label{e:localmollifiedg}\\
\|\tilde u\|_{k+2}&\leq& C\|u\|_2\,\ell^{-k}\leq C\delta\ell^{-(k+1)},
\end{eqnarray}
and
\begin{equation}\label{e:commutate}
\begin{split}
\|\tilde u^\sharp e-\tilde g\|_k&\leq \|\tilde u^\sharp e-(u^\sharp e)*\varphi_\ell\|_k+
\|(u^\sharp e)*\varphi_\ell-g*\varphi_\ell\|_k\\
&\leq \;C\;\ell^{2-k}\|u\|_2^2\;+\;C\;\ell^{-k}\|u^\sharp e-g\|_0\;\;\leq\;\; 
C\; \delta^2\ell^{-k},
\end{split}
\end{equation}
where $k=0,1,\dots,n_*$.
Moreover, since the set $\{h\in\sym_n^+:\,|h-g_0|\leq r\}$ is convex,
$\tilde g$ also satisfies
$\|\tilde g-g_0\|_0\leq r$.

\medskip

{\bf Step 2. Rescaling.} 
First of all, observe that
$$
\tilde h:=\tilde g+\frac{r}{C\delta^2}(\tilde g-\tilde u^\sharp e)
$$
satisfies the condition $|\tilde h(x)-g_0|\leq \frac{r}{C\delta^2}\|\tilde g-\tilde u^\sharp e\|_0 + r\leq 2r$.
Therefore, using Lemma \ref{l:Carat} we have
$$
(1+Cr^{-1}\delta^2)\tilde g\,-\,\tilde u^\sharp e\,=\,\frac{C\delta^2}{r}\,\tilde h=\sum_{i=1}^{n_*}\tilde a_i^2\nu_i\otimes\nu_i,
$$
where $\tilde a_i(x)=\left(C\frac{\delta^2}{r}L_i(\tilde h(x))\right)^{1/2}$. In particular
$\tilde a_i$ is smooth and
\begin{eqnarray*}
\|\tilde a_i\|_k&\leq& C\delta\frac{\|L_i(\tilde h)\|_k}{\|L_i(\tilde h)\|_0^{1/2}}
\;\leq\; C\delta\|\tilde h\|_k\\ 
&\leq& 
C\delta \left(\|\tilde g\|_k+\frac{1}{\delta^2}\|\tilde g-\tilde u^\sharp e\|_k\right)
\;\leq\; C\delta \ell^{-k}
\end{eqnarray*}
for $k=0,1,2,\dots,n_*$ (note that the first inequality is achieved through
interpolation).
Let
$$
u_0=\frac{1}{(1+Cr^{-1}\delta^2)^{1/2}}\tilde u,\quad a_i=\frac{1}{(1+Cr^{-1}\delta^2)^{1/2}}\tilde a_i.
$$
Then we have
$$
\tilde g\,-\,u_0^\sharp e\,=\,\sum_{i=1}^{n_*}a_i^2\nu_i\otimes\nu_i,
$$
with
\begin{eqnarray}
\|\tilde u-u_0\|_1&\leq& C\delta,\label{e:localrescaledu}\\
\|a_i\|_{0}&\leq& C\delta,\\
\|u_0\|_{k+2}+\|a_i\|_{k+1}&\leq& C\delta \ell^{-(k+1)}\,,\label{e:u0estimate}
\end{eqnarray}
for $k=0,1,\dots,n^*$. Notice that the constants above depend
also on $k$, but since we will only use these estimates for $k\leq n_*$, 
this dependence can be suppressed.

Finally, using \eqref{e:commutate} we have
$\|u_0^\sharp e-g_0\|_0\leq r+C\delta^2$, so that 
$\gamma^{-1}I\leq u_0^{\sharp}e\leq \gamma I$, provided $\delta_0$ is sufficiently small.

\medskip

{\bf Step 3. Iterating one-dimensional oscillations.} 
We now apply $n_*$ times successively Proposition \ref{p:step}, with
$$
\ell_j=\ell K^{-j},\;\lambda_j=K^{j+1}\ell^{-1},\;N_j=n_*-j
$$
for $j=0,1,\dots,n_*$. In other words we construct a sequence of immersions $u_j$ such that
$\frac{1}{\gamma}I\leq u_j^\sharp e\leq \gamma I$ and
\begin{equation}\label{e:inductive}
\|u_j\|_{k+2}\leq C\delta\ell^{-(k+1)}_j\quad\textrm{ for }k=0,1,\dots,N_j.
\end{equation}
To see that Proposition \ref{p:step} is applicable, observe that $\lambda_j=K\ell_j^{-1}$. Therefore
it suffices to check inductively the validity of \eqref{e:inductive}. This follows easily 
from \eqref{e:esth3}. The constants will depend on $j$, but this can again be suppressed  because
$j\leq n_*$. 

In this way we obtain the functions $u_1,u_2,\dots,u_{n_*}$
with estimates
\begin{eqnarray*}
\|u_j\|_{2}\leq C\delta\ell^{-1}K^j,\\
\|u_{j+1}^\sharp e-(u_j^{\sharp}e+a_{j+1}^2\nu_{j+1}\otimes \nu_{j+1})\|_0&\leq& C\frac{\delta^2}{\lambda_j\ell_j}=C\delta^2\frac{1}{K},
\end{eqnarray*}
and moreover
\begin{equation*}
\|u_{j+1}-u_j\|_1\leq C\delta.
\end{equation*}
Observe also that $\|u_j^\sharp e-g_0\|_0\leq r+ C\delta^2$, so that, provided $\delta_0$ is sufficiently small, $\gamma^{-1}I\leq u_j^{\sharp}e\leq \gamma I$ for all $j$. 

Thus $v:=u_{n_*}$ satisfies the estimates
\begin{eqnarray*}
\|v^{\sharp}e-\tilde g\|_0&\leq& C\delta^2\frac{1}{K},\\
\|v\|_2&\leq& C \mu\, K^{n_*},\\
\|v-u_0\|_1&\leq &C\delta.
\end{eqnarray*}
The estimates \eqref{e:localerror}, \eqref{e:localC2} 
and \eqref{e:localC1} follow from the above combined with 
\eqref{e:localmollifiedu}, \eqref{e:localmollifiedg} and \eqref{e:localrescaledu}.

\subsection{Stage for general manifolds}\label{s:globalstage}

Given $M$ as in Theorem \ref{t:global} we fix a finite atlas
of $M$ with charts $\Omega_i$ and a corresponding  partition of 
unity $\{\phi_i\}$, so that $\sum \phi_i =1$ and
$\phi_i\in C^\infty_c (\Omega_i)$. Furthermore, 
on each $\Omega_i$ we fix a choice of coordinates.

Using the partition of unity we define the space $C^k(M)$. In particular, let
$$
\|u\|_k:=\sum_{i}\|\phi_iu\|_k.
$$
Similarly, we define "mollification on $M$" via the partition of unity. In other words we fix 
$\varphi\in C^\infty_c (B_1(0))$, and for a function $u$ on $M$ we define
\begin{equation}\label{e:globalmollification}
u*\varphi_\ell \;:=\; \sum_i (\phi_i u)*\varphi_\ell\, .
\end{equation}
It is not difficult to check that the estimates in Lemma \ref{l:mollify} continue to hold
on $M$ with these definitions.

Next, let $g$ be a metric on $M$ as in Theorem \ref{t:global}. Since $M$ is compact and $g$ is continuous, there exists 
$\gamma>0$ such that
\begin{equation}\label{e:gamma}
\frac{1}{\gamma}I\;\leq\;g\;\leq\;\gamma I\quad\textrm{ in }M.
\end{equation}
Moreover, also by compactness, there exists $r_0>0$ such that Lemma \ref{l:Carat} holds with $r=2r_0$ for any $g_0$ satisfying 
$\frac{1}{\gamma}I\leq g_0\leq\gamma I$. Therefore there exists $\rho_0>0$ so that
\begin{equation}\label{e:radius}
\begin{split}
U\subset\Omega_i\textrm{ for some $i$ and }&\textrm{ osc}_U g<r_0\\
\textrm{ whenever }&U\subset M\textrm{ with }\textrm{diam } U<\rho_0.
\end{split}
\end{equation}
Here $\textrm{osc}_Ug$ is to be evaluated in the coordinates of
the chart $\Omega_i$.

In the following we will need coverings of $M$ with the following property:
\begin{definition}[Minimal cover of $M$]
For $\rho>0$ a finite open covering 
$\mathcal{C}$ of $M$ is a minimal cover of diameter $\rho$ if: 
\begin{itemize}
\item the diameter of each $U\in\mathcal{C}$ is less
than $\rho$;
\item $\mathcal{C}$ can be subdivided into $n+1$
subfamilies $\mathcal{F}_i$, each consisting  
of pairwise disjoint sets.
\end{itemize}
\end{definition}
The existence of such coverings is a well-known fact. For the convenience of the reader we give a short proof at the end of this section.

We are now ready to state the 
iteration stage needed for the proof of Theorem
\ref{t:global}. Recall that $\eta_0>0$ is the constant from Proposition \ref{p:step2}.

\begin{proposition}[Stage, global]\label{p:stage2}
Let $(M^n,g)$ be a smooth, compact Riemannian manifold with $g\in C^\beta(M)$, and let
$\mathcal{C}$ be a minimal cover of $M$ of diameter $\rho<\rho_0$, where $\rho_0$ is as in \eqref{e:radius}. There exists $\delta_0>0$ such that, if
$K\geq 1$ and $u\in C^2 (M, \R^m)$ satisfies
\begin{eqnarray}
\|u^\sharp e - g\|_0 &\leq& \delta^2 < \delta_0^2,\label{e:globalerror0}\\
\|u\|_2 &\leq& \mu,\\
{\rm osc}_{U}\,\nabla u &\leq& \eta_0/2\textrm{ for all }U\in\mathcal{C}\label{e:globalosc},
\end{eqnarray}
then there exists $v\in C^2(M, \R^{m})$
with
\begin{eqnarray}
\|v^\sharp e - g\|_0&\leq& C \delta^2 \left(\frac{1}{K}+
\delta^{\beta-2}\mu^{-\beta}\right)\, ,\label{e:globalerror}\\
\|v\|_2&\leq& C \mu\, K^{(n+1)n_*}\, , \label{e:globalC2}\\
\|u - v\|_1&\leq& C \delta\, .\label{e:globalC1}
\end{eqnarray}
The constants $C$ depend only $(M^n,g)$ and $\mathcal{C}$.
\end{proposition}

\subsection{Proof of Proposition \ref{p:stage2}} We proceed
as in the proof of Proposition \ref{p:stage}. Enumerate the covering as $\mathcal{C}=\{U_j\}_{j\in J}$, and
for each $j$ choose a matrix $g_j\in\sym_n^+$ such that
$$
|g(x)-g_j|\leq r_0\,\textrm{ for }x\in U_j.
$$
Furthermore, fix a partition of unity $\{\psi_j\}$ for $\mathcal{C}$ in the sense that $\psi_j\in C^{\infty}_c(U_j)$ and $\sum_j\psi_j^2=1$ on $M$.

\medskip

{\bf Step 1. Mollification.} The mollification step is precisely as in Proposition \ref{p:stage}. We set
$$
\ell =\frac{\delta}{\mu},
$$ 
and let
\begin{equation*}
\tilde u = u*\varphi_\ell,\quad \tilde g = g*\varphi_\ell,
\end{equation*}
where now the convolution is defined in \eqref{e:globalmollification} above.
Then, as before,
\begin{eqnarray}
\|\tilde u-u\|_1&\leq& C\delta,\label{e:globalmollifiedu}\\
\|\tilde g-g\|_0&\leq& C\|g\|_{\beta}\ell^{\beta},\label{e:globalmollifiedg}\\
\|\tilde u\|_{k+2}&\leq& C\delta\ell^{-(k+1)},\\
\|\tilde u^\sharp e-\tilde g\|_k&\leq& C\delta^2\ell^{-k},
\end{eqnarray}
for $k=0,1,\dots,(n+1)n_*$.
In particular, for any $j\in J$ and any $x\in U_j$
$$
|\tilde g(x)-g_j|\leq r_0+C\ell^{\beta}\leq r_0+C\delta_0^{\beta}\leq \frac{3}{2}r_0
$$
provided $\delta_0>0$ is sufficiently small.

\medskip

{\bf Step 2. Rescaling.} 
We rescale the map analogously to Step 2 in Proposition \ref{p:stage}. Accordingly, 
$$
\tilde h:=\tilde g+\frac{r_0}{2C\delta^2}(\tilde g-\tilde u^\sharp e)
$$
satisfies 
$$
|\tilde h(x)-g_j|\leq \frac{r_0}{2C\delta^2}\|\tilde g-\tilde u^\sharp e\|_0 + \frac{3}{2}r_0\leq 2r_0\textrm{ in }U_j.
$$
Therefore, using Lemma \ref{l:Carat} for each $g_j$ and introducing
$$
u_0=\frac{1}{(1+Cr_0^{-1}\delta^2)^{1/2}}\tilde u
$$
we obtain (as in Proposition \ref{p:stage})
$$
\tilde g\,-\,u_0^\sharp e\,=\,\sum_{i=1}^{n_*}a_{i,j}^2\nu_{i,j}\otimes\nu_{i,j}\,\textrm{ in }U_j
$$
for some functions $a_{i,j}\in C^{\infty}(U_j)$ satisfying the estimates
$$
\|a_{i,j}\|_{C^{k+1}(U_j)}\leq C\delta\ell^{-(k+1)}\textrm{ for $j\in J$ and $k=0,1,\dots,(n+1)n_*$}.
$$
In particular, using the partition of unity $\{\psi_j\}$ we obtain
\begin{equation}\label{e:decomposition}
\tilde g\,-\,u_0^\sharp e\,=\,\sum_{j\in J}\sum_{i=1}^{n_*}(\psi_j a_{i,j})^2\nu_{i,j}\otimes\nu_{i,j},
\end{equation}
with
\begin{eqnarray}
\|u-u_0\|_1&\leq& C\delta,\label{e:globalrescaledu}\\
\|\psi_ja_{i,j}\|_0&\leq& C\delta,\\ 
\|u_0\|_{k+2}+\|\psi_ja_{i,j}\|_{k+1}&\leq& C\delta \ell^{-(k+1)}
\end{eqnarray}
for $k=0,1,\dots,(n+1)n_*$.

\medskip

{\bf Step 3. Iterating one--dimensional oscillations}
We now argue as in the Step 3 of the proof of Proposition
\ref{p:stage}. However, there are two
differences. First of all we apply Proposition \ref{p:step2}
in place of Proposition \ref{p:step}. This requires an additional control 
of the oscillation of $\nabla u$ in each $U_j$. 
Second, the number of steps is $(n+1)n_*$.
Indeed, observe that \eqref{e:decomposition} can be written as
\begin{equation*}
\tilde g\,-\,u_0^\sharp e\,=\,\sum_{\sigma=1}^{n+1}\sum_{i=1}^{n_*}\sum_{j\in J_\sigma}(\psi_j a_{i,j})^2\nu_{i,j}\otimes\nu_{i,j},
\end{equation*}
where the index set $J$ is decomposed as $J=J_1\cup\dots\cup J_{n+1}$ so that
$U_j\in\mathcal{F}_\sigma$ if and only if $j\in J_\sigma$. The point is that the sum in $j$ consists of functions with disjoint supports, and hence for this sum Proposition \ref{p:step2} can be performed in parallel, in one step.
Thus, the number of steps to be performed serially is the number of summands in $\sigma$ and $i$, which is precisely $(n+1)n_*$. 

To deal with the restriction on the oscillation of $u_k$ in each step, 
observe that ${\rm osc}_{U_j}\nabla u\leq \eta_0/2$ by assumption, and clearly the same holds for $u_0$. Also, at each step we have the estimate $\|u_{k+1}-u_k\|_1\leq C\delta\leq C\delta_0$. Therefore, choosing
$\delta_0>0$ sufficiently small (only depending on the constants and on $\eta_0$), we ensure that 
the condition remains satisfied inductively $(n+1)n_*$ times. 

Thus, proceeding as in the proof of Proposition \ref{p:stage} we 
apply Proposition \ref{p:step2} successively with $\ell_k=\ell K^{-k}$,  
$\lambda_k=K^{k+1}\ell^{-1}$, and $N_k=(n+1)n_*-k$. In this way we obtain a final map $v:=u_{(n+1)n_*}$ such that
\begin{eqnarray*}
\|v^{\sharp}e-\tilde g\|_0&\leq& C\delta^2\frac{1}{K},\\
\|v\|_2&\leq& C\mu\, K^{(n+1)n_*},\\
\|v-u_0\|_1&\leq &C\delta.
\end{eqnarray*}
The above inequalities combined with \eqref{e:globalmollifiedu}, 
\eqref{e:globalmollifiedg} and \eqref{e:globalrescaledu} imply
the estimates \eqref{e:globalerror}, \eqref{e:globalC2} and \eqref{e:globalC1}. 
This concludes the proof.

\subsection{Existence of minimal covers}

We fix a triangulation $T$ of $M$ with
simplices having diameter smaller than $\rho/3$. 
We let $S_0$ be the vertices of the triangulation,
$S_1$ be the edges, $S_k$ be the $k$--faces. 
$\mathcal{F}_0$ is made by pairwise disjoint balls centered on
the elements of $S_0$, with radius smaller than $\rho/2$. We let
$M_0$ be the union of these balls. Next,
for any element $\sigma\in S_1$, we consider
$\sigma'= \sigma\setminus M_0$. The 
$\sigma'$ are therefore pairwise disjoint compact
sets and we let $\mathcal{F}_1$ be a collection
of pairwise disjoint neighborhoods of $\sigma'$,
each with diameter less than $\rho$. We define
$M_1$ to be the union of the elements of $\mathcal{F}_1$
and $\mathcal{F}_0$. We proceed inductively.
At the step $k$, for every $k$--dim. face $F\in S_k$ 
we define $F'= F\setminus A_{k-1}$. Clearly, the $F'$
are pairwise disjoint compact sets and hence
we can find pairwise disjoint neighborhoods of
the $F'$ with diameter smaller than $\rho$. 
Figure \ref{f:triang} below shows the elements of 
$\mathcal{F}_i$ for a $2$--d triangulation. 

\begin{figure}[htbp]
    \input{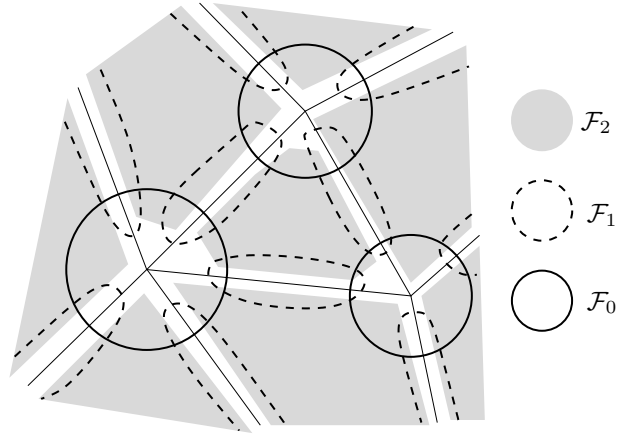}
    \caption{The triangulation $T$ and the covering
for a $2$--dimensional manifold.}
    \label{f:triang}
\end{figure}

Clearly, the collection $\mathcal{F}_0\cup \ldots
\cup \mathcal{F}_n$ covers any simplex of $T$, and hence 
is a covering of $M$.

\section{$h$--principle: iteration}\label{s:iteration}

\subsection{Proof of Theorem \ref{t:local}}
Let $\mu_0,\delta_0>0$ be such that
\begin{eqnarray*}
\|u^\sharp e-g\|_0 &\leq& \delta_0^2\\
\|u\|_2 &\leq& \mu_0.
\end{eqnarray*}
Let also $K\geq 1$. Later on we are going to adjust the parameters $\mu_0$ and $K$ 
in order to achieve the required convergence in $C^{1,\alpha}$.

Applying Proposition \ref{p:stage} successively, we obtain a sequence 
of maps $u_k\in C^2(\overline{\Omega},\R^{n+1})$ such that
\begin{eqnarray*}
\|u_{k}^\sharp e-g\|_0 &\leq& \delta_{k}^2\\
\|u_{k}\|_2 &\leq& \mu_{k}\\
\|u_{k+1}-u_{k}\|_1&\leq& C\delta_{k},
\end{eqnarray*} 
where
\begin{eqnarray}
\delta_{k+1}^2 &=& C 
\delta_{k}^2\left(\frac{1}{K}+\delta_{k}^{\beta-2}\mu_k^{-\beta}\right),
\label{e:will_get_worse}\\
\mu_{k+1} &=& C \mu_k K^{n_*}\label{e:will_get_absorbed}.
\end{eqnarray}
Substituting $K$ with $\max\, \{C^{1/n_*} K , K\}$ we can absorbe the constant
in \eqref{e:will_get_absorbed} to achieve $\mu_{k+1}= \mu_k K^{n_*}$,
at the price of getting a possibly worse constant in \eqref{e:will_get_worse}.
In particular $\mu_k=\mu_0K^{kn_*}$. 
Next, we show by induction that for any
\begin{equation}\label{e:acond}
a\;<\;
\min\left\{\frac{1}{2},\frac{\beta n_*}{2-\beta}\right\}
\end{equation}
there exists a suitable initial choice of $K$ and $\mu_0$ so that
$$
\delta_k\leq \delta_0\,K^{-ak}.
$$
The case $k=0$ is obvious. Assuming the inequality to hold for $k$, we have
$$
\delta_{k+1}^2\leq C 
\delta_0^2K^{-2ak-1}+C\delta_0^{\beta}\mu_0^{-\beta}K^{-\beta k(a+n_*)}.
$$
Therefore $\delta_{k+1}\leq \delta_0 K^{-a(k+1)}$ provided
$$
2C \leq K^{1-2a}\,\textrm{ and }\,2C\leq 
\mu_0^{\beta}\delta_0^{2-\beta}K^{k[\beta(a+n_*)-2a]-2a}.
$$
By choosing first $K$ and then $\mu_0\geq \|u\|_2$ sufficiently large, 
these two inequalities can be satisfied for any given 
$a$ in the range prescribed in \eqref{e:acond}. This proves our claim.

Next we show that for any
\begin{equation}\label{e:alphacond}
\alpha < \min\left\{ \frac{1}{1+2n_*}, \frac{\beta}{2}\right\}
\end{equation}
the parameters $\mu_0$ and $K$ can be chosen so that the sequence $u_k$ converges in $C^{1,\alpha}(\Omega;\R^{n+1})$. To this end observe that to any $\alpha$ satisfying \eqref{e:alphacond} there exists an $a$ satisfying \eqref{e:acond} such that
$$
\alpha<\frac{a}{a+n_*}.
$$
Then, choosing $\mu_0$ and $K$ sufficiently large as above, we obtain a sequence $u_k$ such that
\begin{eqnarray*}
\|u_{k+1}-u_k\|_1 &\leq& C\delta_0\,K^{-ak}\\
\|u_{k+1}-u_k\|_2 &\leq& \mu_{k+1}+\mu_k \leq 2\mu_0K^{(k+1)n_*}.
\end{eqnarray*}
Therefore, by interpolation 
\begin{equation*}
\begin{split}
\|u_{k+1}-u_k\|_{1,\alpha} &\leq  \|u_{k+1}-u_k\|_1^{1-\alpha}\|u_{k+1}-u_k\|_2^{\alpha}\\
& \leq \tilde C\, K^{-[(1-\alpha)a-\alpha n_*]k}.
\end{split}
\end{equation*}
Thus the sequence converges in $C^{1,\alpha}$ to some limit map $v\in C^{1,\alpha}(\overline{\Omega};\R^{n+1})$. Since $\delta_k\to 0$, the limit satisfies $v^\sharp e=g$ in $\Omega$.

Finally, choosing $K$ so large that $K^{-a}\leq 1/2$, we have
$$
\|v-u\|_1\leq C\delta_0\sum_kK^{-ak}\leq 2C\delta_0.
$$

\subsection{Proof of Theorem \ref{t:global}}
Recall from Section \ref{s:globalstage} that for the whole construction we work with a fixed atlas $\{\Omega_i\}$ of the manifold $M$, and that to the given metric $g\in C^\beta(M)$ there exist constants $\gamma>1$ and $\rho_0>0$ such that \eqref{e:gamma} and \eqref{e:radius} hold. 

Since $u\in C^2(M;\R^m)$ and there are a finite number of charts 
$\Omega_i$, there exists $\rho<\rho_0$ such that 
$$
\textrm{ osc}_U\nabla u<\eta_0/4\quad\textrm{ whenever }U\subset M\textrm{ with }\textrm{diam }U<\rho.
$$
Fix a minimal cover $\mathcal{C}$ of $M$ with diameter $\rho$ and
let $\mu_0,\delta_0>0$ be such that
\begin{eqnarray*}
\|u^\sharp e-g\|_0 &\leq& \delta_0^2\\
\|u\|_2 &\leq& \mu_0.
\end{eqnarray*}
The iteration now proceeds with respect to this fixed cover, parallel to the 
proof of Theorem \ref{t:local}. More precisely, arguing as in 
in Theorem \ref{t:local},
Proposition \ref{p:stage2} yields a sequence $u_k\in C^2(M;\R^m)$ with
\begin{eqnarray*}
\|u_{k}^\sharp e-g\|_0 &\leq& \delta_{k}^2\\
\|u_{k}\|_2 &\leq& \mu_0\,K^{k(n+1)n_*}\\
\|u_{k+1}-u_{k}\|_1&\leq& C\delta_{k},
\end{eqnarray*} 
where
\begin{equation*}
\delta_{k+1}^2 = C \delta_{k}^2\left(\frac{1}{K}+\delta_{k}^{\beta-2}
K^{-\beta k(n+1)n_*}\right).
\end{equation*}
The proof that $\mu_0$ and $K$ can be chosen so that $u_k$ converges in $C^{1,\alpha}$ for 
\begin{equation*}
\alpha < \min\left\{ \frac{1}{1+2(n+1)n_*}, \frac{\beta}{2}\right\}
\end{equation*}
follows entirely analogously. Recall that this argument yields in particular
$$
\delta_k\leq \delta_0 K^{-ak}.
$$
The only difference is that the estimates
\eqref{e:globalerror0} and \eqref{e:globalosc} need to be fulfilled at each stage.
To this end note that $\delta_k\leq \delta_0$, so that \eqref{e:globalerror0} will hold at stage $k$ if it holds at the initial stage. Moreover,
$$
{\rm osc}_U\nabla u_k\leq {\rm osc}_U\nabla u + \sum_{j=0}^{k-1} 2\|u_{j+1}-u_j\|_1 \leq
\frac{\eta_0}{4} + 2C\delta_0\sum_jK^{-aj}\leq\frac{\eta_0}{4}+ 4C\delta_0,
$$
so that \eqref{e:globalosc} is fulfilled by $u_k$ provided $\delta_0$ is sufficiently small (depending only on the various constants).

\subsection{Proof of Corollaries \ref{c:local} and \ref{c:global}}\label{ss:emb} 
The corollaries are a direct consequence of the Nash-Kuiper theorem combined with Theorems \ref{t:local} and \ref{t:global} respectively. For simplicity, we allow $M$ to be
either $\overline{\Omega}$ for a smooth bounded open set 
$\Omega\subset\R^n$ or a compact Riemannian manifold
of dimension $n$, and assume that $g\in C^{\beta}(M)$ is satisfying either
the assumptions of Theorem \ref{t:local}
or those of Theorem \ref{t:global}. We then
set $\alpha_0 = \min\{(2n_*+1)^{-1},\beta/2\}$ in the first
case, and $\alpha_0 = \min \{(2(n+1)n_*+1)^{-1},\beta/2\}$ in the second.

Let $u\in C^1(M;\R^m)$ be a short map and $\e>0$. We may assume without loss of generality that $\e<\delta_0$.
Using the Nash-Kuiper theorem together with a standard regularization, there exists $u_0\in C^2(M;\R^m)$
such that 
\begin{eqnarray*}
\|u-u_0\|_1 &\leq& \e/2,\\
\|u_0^{\sharp}e - g\|_0 &\leq& \left(\frac{\e}{2C}\right)^2,
\end{eqnarray*}
where $C$ is the constant in Theorems \ref{t:local} and \ref{t:global} respectively. Then the theorem, applied to $u_0$, yields an isometric immersion $v\in C^{1,\alpha}(M;\R^m)$ for any $\alpha<\alpha_0$, such that 
$\|v-u_0\|_1\leq \e/2$, so that $\|v-u\|_1\leq \e$. This proves the corollaries.

We now come to Remark \ref{r:emb}. This follows immediately from the fact that the Nash-Kuiper theorem also works for embeddings, and that the set of embeddings of a compact manifold is an open set in $C^1(M;\R^m)$. Indeed, if $u$ is an embedding, the Nash-Kuiper theorem gives the existence of an embedding $u_0$  with the estimates above. Ensuring in addition that $\e$ is so small that any map $v\in C^1(M;\R^m)$ with $\|v-u\|_1\leq\e$ is an embedding, we reach the required conclusion.

\section{Rigidity: Proof of Theorem \ref{t:2/3}}\label{s:rigidity}

\subsection{Curvature and Brouwer degree}
Let $(M,g)$ be as in Theorem \ref{t:2/3}. 
As usual, we denote by $dA$ the area element in $M$
and by $\kappa$ the Gauss curvature of $(M,g)$. 
Consider next a $C^2$ isometric embedding $v: M \to \R^3$.
The unit normal $N (p)$ to $v (M)$ is the unique vector 
of $\R^3$ such that, given a positively oriented 
basis $e_1, e_2$ for $T_p (M)$, 
the triple $(dv_p (e_1), dv_p (e_2), N (p))$ is an
orthonormal positively oriented frame of $\R^3$.

As it is well known, if $d\sigma$ denotes the
area element in $\S^2$, then $N^\sharp d\sigma = \kappa dA$.
Therefore, for every open set $V\subset\subset M$ and
for every $f \in C^1 (\S^2)$, the usual change of variable
formula yields
\begin{equation}\label{e:ChangeOfVar}
\int_V f(N(x)) \kappa (x)\, dA(x)\;=\;
\int_{\S^2}f(y)\deg(y,V,N)\, d\sigma(y),
\end{equation}
where $\deg (y, V, N)$ denotes the Brouwer
degree of the map $N$. Though the differential definition
of $\deg$ makes sense only for regular values of $N$,
it is a classical observation that $\deg$ is constant on
connected components of $\S^2\setminus N (\partial V)$.
Thus it has a unique continuous extension to 
$\S^2\setminus N (\partial V)$, which will be denoted as well
by $\deg$.

Consider next an isometric embedding $v\in C^1$.
In this case $N\in C^0$. The Brouwer degree
$\textrm{\deg} (y, V, N)$ can still be defined and we recall 
the following well-known theorem.

\begin{theorem}\label{t:deg}
Let $N\in C (V, \S^2)$ and $\{N_k\} \subset
C^\infty (V, \S^2)$ be a sequence converging uniformly to $N$.
Let $K \subset \S^2\setminus N (\partial V)$ be a closed set. 
For any $k$ sufficiently large, $\deg (\cdot, V, N_k)
\equiv \deg (\cdot, V, N)$ on $K$. 
\end{theorem}

Thus $\deg (\cdot , V, N)\in L^1_{loc} (\S^2\setminus N (\partial V))$.
A key step to the proof of Theorem \ref{t:2/3} is to show
that formula \eqref{e:ChangeOfVar} 
holds for $v\in C^{1,\alpha}$ with $\alpha>2/3$.

\begin{proposition}\label{p:ChangeOfVar}
Let $v\in C^{1,\alpha} (M, \R^3)$ be
an isometric embedding with $\alpha>2/3$.
Then \eqref{e:ChangeOfVar} holds for every open set
$V\subset\subset M$ diffeomorphic to a subset of $\R^2$ 
and every $f\in L^\infty$ with $\supp (f)\subset
\S^2\setminus N (\partial V)$. 
\end{proposition}

In order to deal with $N (\partial V)$ we recall
the following elementary fact. 

\begin{lemma}\label{l:nullset}
Let $M$ and $\tilde M$ be $2$-dimensional Riemannian manifolds
and $N\in C^{0,\beta} (M, \tilde M)$ with $\beta > 1/2$. If
$E\subset M$ has Hausdorff dimension $1$, 
then the area of $N(E)$ is $0$.   
\end{lemma}

The following is then a corollary of Proposition \ref{p:ChangeOfVar}
and Lemma \ref{l:nullset}.

\begin{corollary}\label{c:ChangeOfVar2}
Let $(M,g)$ and $v$ be as in Proposition \ref{p:ChangeOfVar},
with $\kappa\geq 0$. For any open $V\subset \subset M$,
$\deg (\cdot, V, N)$ is a nonnegative $L^1$ function and 
\eqref{e:ChangeOfVar} holds for every $f\in L^\infty (\S^2\setminus
N (\partial V))$. 
\end{corollary}

\subsection{Proof of Proposition \ref{p:ChangeOfVar}}

By a standard approximation argument, it suffices to prove
the statement when $f$ is smooth. Under this additional assumption
the proof is a direct consequence of Theorem \ref{t:deg} and
of the convergence result below, which is a consequence
of Proposition \ref{p:improv}.
Since $V$ is 
diffeomorphic to an open set of the euclidean plane,
we can consider global coordinates $x_1, x_2$ on it.
Fix a symmetric kernel $\varphi\in C^\infty_c (\R^2)$,
set $\varphi_\e (x) = \e^{-2} \varphi (x/\e)$ and let $v^\e:=
(v{\bf 1}_V) * \varphi_\e$ (we consider
here the convolution of the two functions in $\R^2$ using
the coordinates $x_1,x_2$ and the corresponding Lebesgue measure).

\begin{proposition}\label{p:conv}
Let $v$ and $v^\e$ be defined as above
and denote by $N^\e$,
$g^\e$, $A^\e$ and $\kappa^\e$ respectively, the normal
to $v^\e (M)$, the pull-back of the metric
on $v^\e (M)$, and the corresponding area element
and Gauss curvature. Then,
\begin{equation}\label{e:conv}
\lim_{\e\downarrow 0} \int_V f (N^\e)\kappa^\e\, dA^\e
\;=\; \int_V f(N) \kappa\, dA\, \quad
\forall f\in C^\infty_c (\S^2\setminus N (\partial V))\, .
\end{equation} 
\end{proposition}

\begin{proof}[Proof of Proposition \ref{p:conv}] 
In coordinates, our aim is to show
that
\begin{equation}\label{e:goal}
\lim_{\e\downarrow 0} \int_V f (N^\e (x))
\kappa^\e (x)\, (\det g^\e (x))^{\frac{1}{2}}\, dx
\;=\; \int_V f (N (x)) \kappa (x)\, (\det g (x))^{\frac{1}{2}}\, dx\, .
\end{equation}
We recall the formulas for the Christoffel symbols, 
the Riemann tensor and the Gauss curvature in 
$V$, in the system of coordinates already fixed:
\begin{eqnarray}
\Gamma_{jk}^i&=&\frac{1}{2}g^{im}\bigl(\partial_kg_{jm}+
\partial_jg_{mk}-\partial_mg_{kj}\bigr),\label{e:Christoffel}\\
R_{iljk}&=&g_{lm}\bigl(\partial_k\Gamma_{ij}^m-\partial_j
\Gamma_{ik}^m+\Gamma_{ij}^l\Gamma_{kl}^m-\Gamma_{ik}^l
\Gamma_{jl}^m\bigr),
\label{e:Riemannian}\\
\kappa &=&\frac{R_{1212}}{\det(g_{ij})}.\label{e:Gaussian}
\end{eqnarray}
After obvious computations we conclude that 
\begin{equation}\label{e:expression}
\kappa \;=\; (\det g)^{-1} \left(c_{ijkl}\, \partial_{kl} g_{ij}
+ d_{ijklmn} (g)\, \partial_k g_{ij}\, \partial_l g_{mn}\right)
\end{equation}
where $c_{ijkl}$ are {\em constant} coefficients and the functions
$d_{ijklmn}$ are smooth. 

Proposition \ref{p:improv} implies that $\partial_k g^\e_{ij}$ and
$g^\e_{ij}$ converge locally uniformly to 
$\partial_k g_{ij}$ and $g_{ij}$ respectively. Moreover,
$N^\e$ converges locally uniformly to $N$. Since
there is a compact set containing $f (N^\e)$ and $f(N)$,
we only need to show that
\begin{eqnarray}
&&\lim_{\e\downarrow 0} \int_V f (N^\e (x))
(\det g^\e (x))^{-\frac{1}{2}}\, \partial_{kl} g^\e_{ij} (x)\, 
dx\nonumber\\ 
&=& \int_V f (N (x)) (\det g (x))^{-\frac{1}{2}}\, \partial_{kl}
g_{ij} (x)\, dx\, .\label{e:goal2}
\end{eqnarray}
Denote by $\psi^\e$ the function
$f (N^\e (x)) (\det g^\e (x))^{-\frac{1}{2}}$. 
Since $f (N^\e)$ is smooth and compactly supported in $V$ we
can integrate by parts to get
\begin{equation}\label{e:perparti}
\int_V \psi^\e \partial_{kl} g^\e_{ij}\;
=\; \int_V \partial_k \psi^\e \partial_l g^\e_{ij}\, \, .
\end{equation}
Note that $\|\partial_k \psi^\e\|\leq C \e^{\alpha-1}$
by obvious estimates on convolutions. Hence, \eqref{e:improv} gives
\begin{equation}\label{e:perparti2}
\int_V \partial_k \psi^\e \left(\partial_l g^\e_{ij} -
\partial_l g_{ij}\right)\;=\; O (\e^{3\alpha-2})\, 
\end{equation}
which converges to $0$ because $\alpha>3/2$.
Integrating again by parts, we get
\begin{eqnarray*}
&&\lim_{\e\downarrow 0} \int_V f (N^\e (x))
(\det g^\e (x))^{-\frac{1}{2}}\, \partial_{kl} g^\e_{ij} (x)\, 
dx\\
&=& \lim_{\e\downarrow 0} \int_V f (N^\e (x))
(\det g^\e (x))^{-\frac{1}{2}}\, \partial_{kl} g_{ij} (x)\, 
dx\, .
\end{eqnarray*}
Using the uniform convergence of $N^\e$ to $N$ and
of $g^\e$ to $g$ we then
conclude \eqref{e:goal2} and hence the proof of the
Proposition.
\end{proof}

\subsection{Proof of Lemma \ref{l:nullset} and Corollary
\ref{c:ChangeOfVar2}}

\begin{proof}[Proof of Lemma \ref{l:nullset}]
By the definition of Hausdorff dimension, for every $\e >0$
and $\eta>1$ 
there exists a covering of $E$ with closed sets $E_i$ such that
\begin{equation}
\sum_i (\textrm{diam}\, (E_i))^\eta \;\leq\; \e\, .
\end{equation} 
On the other hand, $\textrm{diam}\, (g (E_i)) \leq C 
(\textrm{diam} (E_i))^\beta$ 
and hence the area $|g(E_i)|$ can be estimated with 
$C (\textrm{diam}\, (E_i))^{2\beta}$. 
Since $\beta>1/2$, we can pick $\eta =
2\beta$ to conclude that 
$$
|g(E)|\;\leq\; C \sum_i (\textrm{diam}\, (E_i))^{\eta}\;\leq\;
C \e\, .
$$
The arbitrariness of
$\e$ implies $|g(E)|=0$.
\end{proof}

\begin{proof}[Proof of Corollary \ref{c:ChangeOfVar2}]
First of all, we know from Proposition
\ref{p:ChangeOfVar} that the formula \eqref{e:ChangeOfVar}
is valid for any open set $V$ which is diffeomorphic to an
open set of $\R^2$, and any $f\in L^\infty$
compactly supported in $\S^2\setminus N (\partial V)$.
Since $\kappa$ is nonnegative, we 
conclude that $\deg (\cdot, N, V)\geq 0$.
Testing \eqref{e:ChangeOfVar} with a sequence
of compactly supported functions $f_k\uparrow {\bf 1}_{\S^2\setminus
N (\partial V)}$
we derive that
$$
\int \deg (y, N, V)\, d\sigma (y) \;=\;
\int_V \kappa\, dA \;<\; \infty\, ,
$$ 
which implies $\deg (\cdot, N, V)\in L^1$.

\medskip

Next, consider a $V$ with smooth boundary.
We decompose it into the union of finitely many 
nonoverlapping Lipschitz open sets $V_i$ diffeomorphic to
open sets of the euclidean plane. Then
$$
\deg (y, N, V) \;=\; \sum_i \deg (y, N, V_i)
\qquad \mbox{for every $y\not\in \bigcup N (\partial V_i)$.}
$$
On the other
hand, by Lemma \ref{l:nullset}, $\bigcup_i N (\partial V_i)$
is a negligible set, and hence we conclude the formula
for $V$ from the previous step.

\medskip

Finally, fix a generic $V$ and an $f\in L^\infty$
with $\supp (f) \subset \S^2 \setminus N (\partial V)$. Choose
an open set $V'$ with smooth boundary
$\partial V'$ sufficiently close to $\partial V$.
Then $\deg (\cdot, V, N)$ and $\deg (\cdot, V', N)$ coincide
on the support of $f$, whereas the support of $f (N (\cdot))$
is contained in $V'$. From the formula for $V'$ and $f$
we conclude then the validity of the formula for $V$ and $f$.
Arguing again as above, we conclude that
$\deg (\cdot, N, V)$ is summable and nonnegative and that
the formula \eqref{e:ChangeOfVar} holds for any $V$ and any 
$f\in L^\infty (\S^2\setminus N (\partial V))$.
\end{proof}

\subsection{Bounded extrinsic curvature. The proof of Theorem \ref{t:2/3}}
We recall the notion of bounded extrinsic curvature for
a $C^1$ immersed surface
(see p. 590 of \cite{Pogorelov73}). 

\begin{definition}\label{d:b_ex}
Let $\Omega\subset \R^2$ be open and $u\in C^1 (\Omega,
\R^3)$ an immersion. 
The surface $u (\Omega)$ has
{\em bounded extrinsic curvature} if there is a $C$ such that
\begin{equation}\label{e:b_ex}
\sum_{i=1}^N |N (E_i)| \;\leq\; C\qquad
\end{equation}
for any finite collection $\{E_i\}$ of pairwise disjoint closed
subsets of $\Omega$.
\end{definition}

The proof of Theorem \ref{t:2/3} follows now from Corollary \ref{c:ChangeOfVar2}.

\begin{proof}[Proof of Theorem \ref{t:2/3}] 
The theorem follows easily from the claim:
\begin{equation}\label{e:bigger}
\deg (\cdot , V, N)\;\geq\; {\bf 1}_{N (V)\setminus
N (\partial V)} \qquad\mbox{for every open $V\subset \Omega$.}
\end{equation}
In fact, given disjoint closed sets $E_1, \ldots, E_N$,
we can cover them with disjoint open sets $V_1, \ldots V_N$
with smooth boundaries. 
By \eqref{e:bigger} and Corollary \ref{c:ChangeOfVar2},
\begin{equation}\label{e:bec}
\sum_i |N(E_i)\setminus N (\partial V_i)|
\;\leq\; \sum_i |N (V_i)\setminus N (\partial V_i)|
\;\leq\; \sum_i \int_{V_i} \kappa \;\leq\; \int_\Omega \kappa\, .
\end{equation}
On the other hand, by Lemma \ref{l:nullset}, $|N (\partial V_i)|=0$.
Thus, \eqref{e:bec} shows \eqref{e:b_ex}.

\medskip

We now come to the proof of \eqref{e:bigger}.
Obviously $\deg (y, V, N) = 0$ if $y\not\in N(V)$.
Moreover, by Corollary \ref{c:ChangeOfVar2}, 
$\deg (\cdot, V, N)\geq 0$.
Therefore, fix $y_0\in N(V)\setminus N (\partial V)$ and assume,
by contradiction,
that $\deg (y_0, V, N)=0$. 
Consider a small open disk $D$ centered
at $y_0$ such that $N^{-1} (D) \cap \partial V =\emptyset$
and let $W:= N^{-1} (D)\cap V$. Then $N(\partial W) \subset
\partial D$ and $N(W)\subset D$. 
So, $\deg (\cdot, W, N)$ vanishes
on $\S^2\setminus
\overline{D}$ and is a constant integer $k$ on $D$. 
On the other hand $k = \deg (y_0, W, N) = \deg (y_0, V, N) -
\deg (y_0, V\setminus \overline{W}, N) = - \deg (y_0, V\setminus
\overline{W}, N)$. Since $y_0\not \in N (V\setminus
\overline{W})$, we conclude $k=0$ and hence
$$
0 \;=\; \int \deg (y, W, N)\, dy \;=\;
\int_W \kappa dA\, .
$$
which is a contradiction becase $W\neq \emptyset$ and $\kappa>0$.
\end{proof}

Corollary \ref{c:2/3} 
follows from 
Theorem \ref{t:2/3} and the results of 
Pogorelov cited in the introduction. 
More precisely, by Theorem 9 on p650 
\cite{Pogorelov73}, $u(S^2)$ is a closed convex surface,
which by \cite{PogorelovRigidity} is rigid. 

Corollary \ref{c:localconvexity} also follows from the results in 
\cite{Pogorelov73} and \cite{Sabitov}. However, we were unable to find an
exact reference for open surfaces, and therefore, for the reader's
convenience, we have included
a proof in the appendix.

\appendix
\section{Proof of Corollary \ref{c:localconvexity}}

First of all, since the theorem is local,
without loss of generality we can assume that:
\begin{itemize}
\item $\Omega = B_r (0)$, 
$u\in C^{1, \alpha} (\overline{B}_r (x))$,
$g\in C^{2,\beta} (\overline{B}_r (x))$ and $u$ is an embedding;
\item $u (\Omega)$ has bounded extrinsic curvature.
\end{itemize}

\medskip

{\bf Step 1. Density of regular points.}
For any point $z\in \S^2$
we let $n (z)$ be the cardinality of $N^{-1} (z)$.
It is easy to see that,
for a surface of bounded extrinsic curvature, 
$\int_{\S^2} n <\infty$
(cp. with Theorem 3 of p. 590 in \cite{Pogorelov73}).
Therefore, the set $E:=\{n=\infty\}$ has
measure zero. Let $\Omega_r:= N^{-1} (\S^2\setminus E)$. 
Observe that
\begin{equation}\label{e:density}
\mbox{$\Omega_r$ is dense in $\Omega$.}
\end{equation}
Otherwise
there is a nontrivial smooth open set $V$ such that $N(V)\subset
E$. But then, $\deg (\cdot, V, N)=0$ for every $y\not\in N
(\overline{V})$, and since $|N(V)|= |N(\partial V)|=0$,
it follows that $\deg (\cdot, V, N)=0$ a.e.. By
Corollary \ref{c:ChangeOfVar2},
$\int_V \kappa = 0$,
which contradicts $\kappa>0$. 

\medskip

{\bf Step 2. Convexity around regular points.}
Note next that, for every $x\in \Omega_r$  there is
a neighborhood $U$ of $x$ such that $N(y)\neq N(x)$ for
all $y\in U\setminus \{x\}$, i.e. $x$ is
regular in the sense of \cite{Pogorelov73} p. 582. Recalling
\eqref{e:bigger},
$\deg (\cdot , V, N)\geq {\bf 1}_{V\setminus
\partial V}$ for every $V$: therefore
the index of the map $N$ at every
point $x\in \Omega_r$ is at least $1$. So, by the Lemma
of page 594 in \cite{Pogorelov73}, any point $x\in \Omega_r$
is an elliptic point relative to the mapping $N$
(that is, there is a neighborhood $U$ of $x$ such that
the tangent plane $\pi$ to $u(\Omega)$ in $x$ intersects
$U\cap u (\Omega)$ only in $u(x)$; cp. with page 593 of
\cite{Pogorelov73}). 

By the discussion of page 650 in \cite{Pogorelov73},
$u (\Omega)$ has nonnegative extrinsic
curvature as defined in IX.5 of \cite{Pogorelov73}. Then,
Lemma 2 of page 612 shows that, for every elliptic
point $y\in u (\Omega)$ there is a neighborhood where
$u (\Omega)$ is convex. This conclusion applies,
therefore, to any $y\in \Omega_r$. We next claim the
existence of a constant $C$ with the following property.
Set $\rho (y):= C^{-1}\min \{1, \dist (u(y), u (\partial \Omega)\}$.
Then
\begin{equation}\label{e:uniform}
\mbox{$u(\Omega)\cap B_{\rho (y)} (y)$ is convex for all 
$y\in \Omega_r$.}
\end{equation}
Recall that $u$ is an embedding and hence
$\dist (u(y), u (\partial \Omega))>0$ for every $y\in \Omega$.
By \eqref{e:density}, \eqref{e:uniform} gives for any
$y\in \Omega$ there is a neighborhood where $u(\Omega)$ is convex.
This would complete the proof.

\medskip

{\bf Step 3. Proof of \eqref{e:uniform}.}
First of all, since $u$ is an embedding and $\|u\|_{C^{1,\alpha}}$
is finite, there is a constant $c_0$ such that, for any point $x$,
$B_{c_0} (x)\cap u (\Omega)$ is the graph of a $C^{1,\alpha}$
function with $\|\cdot\|_{C^{1,\alpha}}$ norm smaller than
$1$. 
In order to prove \eqref{e:uniform} we assume, without loss
of generality, that $y=0$ and that the tangent plane to
$u(\Omega)$ at $y$ is $\{x_3=0\}$. 
Denote by $\pi$ the projection on
$\{x_3=0\}$. By \cite{Sabitov}
there is a constant $\lambda>0$ (depending
only on $\|g\|_{C^{2,\beta}}$, $\|\kappa\|_{C^0}$ and 
$\|\kappa^{-1}\|_{C^0}$) with the following property.
\begin{itemize}
\item[(Est)] Let $U$ be an open convex set such that
$U\cap u(\partial \Omega)=\emptyset$, 
${\rm diam}\, (U) \leq c_0$
and $U\cap u (\Omega)$ is locally convex. 
Then $U\cap u (\Omega)$ is the graph of 
a function $f: \pi (u(\Omega)\cap U)\to \R$ with 
$\|f\|_{C^{2, 1/2}}\leq \lambda^{-1}$ and $D^2 f \geq \lambda {\rm Id}$.
\end{itemize}
\begin{figure}[htbp]
    \input{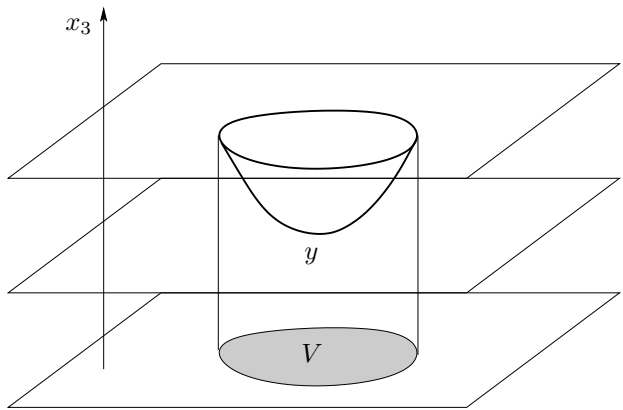}
    \caption{The convex sets of type $V\times ]-a,a[$
among which we choose the maximal one $U_m$.}
    \label{f:maximal}
\end{figure}
We now look for sets $U$ as in (Est) with the additional
property that $U=V\times ]-a, a[$ and $f|_{\partial V}
= a$ (see Figure \ref{f:maximal}).
Let $U_m$ be the maximal
set of this form for which the assumptions of (Est) hold.
We claim that, either $\partial U_m \cap u (\partial \Omega)\neq
\emptyset$, or ${\rm diam}\, (U_m) = c_0$. By (Est),
this claim easily implies \eqref{e:uniform}.
To prove the claim, assume by contradiction that it
is wrong and let $U_m = W_m \times ]-a_m, a_m[$ be the maximal set. 
Let $\gamma= \partial U_m \cap u(\Omega)$. By the choice
of $c_0$,
$\gamma$ is necessarily the curve $\partial W_m\times \{a\}$.
On the other hand, by the estimates of (Est), it follows
that every tangent plane to $u (\Omega)$ at a point of $\gamma$
is transversal to $\{x_3=0\}$. So, for a sufficiently small
$\e>0$, the intersection $\{x_3=a_m+\e\}\cap u (\Omega)$
contains a curve $\gamma'$ bounding a connected region $D\subset 
u(\Omega)$ which contains $u (\Omega)\cap U_m$. By Theorem 8
of page 650 in \cite{Pogorelov73}, $D$ is a convex set. This easily
shows that $U_m$ was not maximal.

\bibliographystyle{acm}
\bibliography{iso}

\end{document}